\documentclass{amsart}
\usepackage{amsmath, amsthm, amssymb}
\usepackage{amsmath,amscd}
\usepackage{mathabx}
\usepackage[hidelinks]{hyperref}
\usepackage{mathrsfs}
\usepackage{xcolor, changepage} 
\usepackage{graphicx}
\usepackage{titletoc}
\usepackage{enumerate}
\usepackage{accents}

\usepackage{caption}
\usepackage{subcaption}

\usepackage{geometry}
\theoremstyle{plain}
\newtheorem{theorem}{Theorem}
\newtheorem{proposition}[theorem]{Proposition} 
\newtheorem{lemma}[theorem]{Lemma}
\newtheorem{remark}[theorem]{Remark}
\newtheorem{corollary}[theorem]{Corollary}
\newtheorem{definition}[theorem]{Definition}
\newtheorem{example}[theorem]{Example}
\newtheorem{prob}[theorem]{Problem}
\numberwithin{theorem}{section}

\numberwithin{equation}{section}

\title[Twisted $S^1$ stability and PSC obstruction on fiber bundles]{Twisted $S^1$ stability and positive scalar curvature obstruction on fiber bundles}
\author{Shihang He}
\address{ School of Mathematical Sciences, Key Laboratory of Pure and Applied Mathematics, Peking University, Beijing, 100871, China}
\email{hsh0119@pku.edu.cn}


\begin{document}

	\begin{abstract}
		We establish several non-existence results of positive scalar curvature (PSC) on fiber bundles. We show that under an incompressible condition of the fiber, for $X^m$ a Cartan-Hadamard manifold or an aspherical manifold when $m=3$, the fiber bundle over $X^m\#M^m$ with the $K(\pi,1)$ fiber, $NPSC^+$(a manifold class including enlargeable and Schoen-Yau-Schick ones) fiber, or spin fiber of the non-vanishing Rosenberg index carries no PSC metric, with necessary dimension and spin compatible conditions imposed. Furthermore, we show that under a homotopically nontrivial condition of the fiber, the $S^1$ bundle over a closed 3-manifold admits a PSC metric if and only if its base space does. These partially answer a question of Gromov and extend some previous results of Hanke, Schick and Zeidler concerning PSC obstruction on fiber bundles.
	\end{abstract}
	
		\maketitle
	
	\section{Introduction}
	
	The $S^1$ factor plays important roles in the study of geometry and topology of manifolds with positive scalar curvature (PSC). The famous Georch conjecture, resolved in \cite{ref24} by Gromov-Lawson and in \cite{ref21} by Schoen-Yau, states $T^n$, completely constructed by $S^1$ factors, is such a typical manifold which admits no PSC metric. One of the major advantages of the $S^1$ factor is that it can be used in the descending argument combined with the minimal hypersurface technique. The descending argument was first proposed by Schoen-Yau in \cite{ref21}. They observed that a stable minimal hypersurface in the codimension 1 homology class of a PSC manifold must have positive Yamabe quotient. This codimension 1 homology class is always assigned to be the Poincare Dual of a 1-dimensional cohomology class, which can be regarded as an $S^1$ factor in an abstract sense. Later in \cite{ref25}, Schick noted that taking cup products of such an $S^1$ factor $(n-2)$ times will yield a complete topological condition to obstruct the existence of a PSC metric. By this observation, he gave a counterexample to the unstable Gromov-Lawson-Rosenberg conjecture. The construction in \cite{ref25} was later generalized to the notion of the Schoen-Yau-Schick (\textit{SYS}) manifold.
	
	The $S^1$ factors are also widely used in geometric problems related to PSC. It was first observed in \cite{ref4} by Fischer-Colbrie and Schoen that an appropriate warped product of a stable minimal hypersurface in a PSC manifold with $S^1$ carries a PSC metric. Later in \cite{ref6}, Gromov and Lawson generalized this argument to a torical symmetrization technique, obtaining Lipschitz estimates for non-zero degree maps from PSC manifolds to the Euclidean unit sphere. More recently, Gromov's seminal work \cite{ref7} revealed how this symmetrization approach leads to powerful width estimates for Riemannian bands, and many breakthroughs have been made following this strategy recently (For example, see \cite{ref3}\cite{ref9}\cite{ref26}\cite{ref27}).
	
	In \cite{ref18}, Rosenberg made the following stability conjecture for $S^1$ multiplication: {\it a closed manifold $M$ admits no PSC metric if and only if $M\times S^1$ admits no PSC metric, provided the dimension of $M$ is not equal to 4,} where the requirement of dimension not being $4$ is due to a counterexample coming from the Seiberg-Witten invariant (c.f.\cite{ref19}). In \cite{ref16}, by using the $\mu$-bubble argument, R{\"{a}}de completely solved this conjecture in dimensions no greater than 7 by establishing a width inequality for Riemannian bands obtained from the product of a general closed manifold and a closed interval. In fact, the following result is proved in \cite[Corollary 2.25]{ref16}: {\it Let $n\in\lbrace 2,3,4,6,7\rbrace$ and $Y^{n-1}$ be a closed connected oriented manifold which does not admit a metric of PSC. If $X = Y\times\lbrack -1,1\rbrack$ and $g$ is a Riemannian metric on $X$ with $Sc(X,g)\ge n(n-1)$, then $width(X,g)<\frac{2\pi}{n}$.} By using this, one can  verify Rosenberg's stability conjecture in dimensions no greater than $7$.
	
	However, all of the above constructions only make sense in the product case, that is, the trivial $S^1$ bundle case. In 1983,  B{\'{e}}rard-Bergery studied the geometry of manifolds with $S^1$-equivariant PSC metrics in \cite{ref1}. Among other things, he got a necessary and sufficient condition for the existence of an $S^1$-invariant PSC metric on a compact manifold with a free $S^1$-action (See Theorem C of \cite{ref1}).  As a direct corollary, $S^1$ bundles over a closed manifold with a PSC metric must admit a PSC metric. With those facts in mind, it is natural to ask the problems:
	\begin{prob}\label{mainprob}
		Let $B$ be a closed manifold which admits no PSC metric, and $E$ be an $S^1$-bundle over $B$. When does $E$ admit no PSC metric? More generally, if $E$ is a fiber bundle with the fiber $F$, where $F$ admits no PSC metric, when does $E$ admit no PSC metric?	
	\end{prob}
	Clearly, Problem \ref{mainprob} generalizes the Rosenberg stability conjecture, and its answer may not always be affirmative. To go further, we introduce the following conception.
	
	\begin{definition}\label{defn: twisted S^1 property}
		$\quad$
		\begin{itemize}
			\item Let $B$ be a closed manifold that admits no PSC metric. We say that $B$ has the \textsl{twisted $S^1$ stability property} if every $S^1$ bundle over $B$ carries no Riemannian metric of PSC.
			\item Let $B$ be a closed manifold that admits no PSC metric. We say that $B$ has the \textsl{dominated twisted $S^1$ stability property} if for every $S^1$ bundle $E$ over $B$ and every closed manifold $E'$ that admits a degree-$1$ map to $E$, $E'$ carries no Riemannian metric of PSC.
		\end{itemize}
	\end{definition}
	
	As a consequence of Definition \ref{defn: twisted S^1 property}, we have the following more specific problem:
	
	\begin{prob}\label{proba}
		Let $B$ be a closed manifold, when does it have the twisted $S^1$ stability property?
	\end{prob}
	
	In \cite{ref10} and \cite{ref22}, by developing a series of equivariant surgery theorems and equivariant bordism theorems, several existence results of PSC on $S^1$ manifolds are established. However, to our knowledge, there are not so many results concerning the non-existence of PSC on non-trivial bundles. The following examples present some progress in Problem \ref{mainprob}.
	
	\begin{example}\label{eg1}
		$\quad$
		
		(a) (\cite{HPS15}, Corollary 4.5) Let $E$ be an $F$-bundle over $B$, where $E$ is spin, $F$ has non-vanishing Rosenberg index and $B$ is a closed surface with $B\ne S^2, \mathbb{R}\mathbb{P}^2$. Then $E$ has non-vanishing Rosenberg index and hence carries no PSC metric.
		
		(b) (\cite{Zei}, Theorem 1.5) Let $E$ be an $F$-bundle over $B$, where $E$ is spin, $F$ has non-vanishing Rosenberg index and $B$ is an aspherical manifold whose fundamental group has finite asymptotic dimension. Then $E$ has non-vanishing Rosenberg index and hence carries no PSC metric.
		
		(c) (\cite{HS06}, Proposition 6.1) Let $E$ be an $F$-bundle over $B$, where $E$ is spin, $F$ and $B$ are enlargeable spin manifolds of even dimension. If the short exact sequence
		\begin{align*}
			0\longrightarrow\pi_1(F)\longrightarrow\pi_1(E)\longrightarrow\pi_1(M)\longrightarrow 0
		\end{align*}
		splits as
		\begin{align*}
			0\longrightarrow\pi_1(F)\longrightarrow\pi_1(F)\times\pi_1(B)\longrightarrow\pi_1(B)\longrightarrow 0,
		\end{align*}
		then $E$ has non-vanishing Rosenberg index and hence carries no PSC metric.
		
		(d) Let $E$ be a fiber bundle over $B$ with the fiber $F$ and $dimE\le 5$, where both $B$ and $F$ are $K(\pi,1)$ manifolds. Then $E$ admits no $PSC$ metric. In particular, the conclusion is  true for the case that $F = S^1$.
		
		(e) (\cite{ref24}, Proposition 4.3, \cite{ref7}, p.658, Example(c)) Let $E$ be a fiber bundle over $B$ with the fiber $F$, where both $B$ and $F$ are Cartan-Hadamard manifolds. Then $E$ is enlargeable and hence admits no PSC metric. In particular, the conclusion is  true for the case that $F = S^1$.
		
	\end{example} 
	
	Among these, (a)(b)(c) are obtained from the index theory, where (b) is a generalization of (a). (d) is a consequence of the recent progress in the aspherical conjecture in dimension no greater than 5 in \cite{ref3} and \cite{ref9}. Indeed, the total spaces in both (d) and (e) are aspherical. Therefore, the answer to Problem \ref{proba} should be positive if the base manifold is an Cartan-Hadamard manifold or an aspherical manifold of dimensions no greater than $4$. However, the following example in (see \cite[Example 9.2]{ref1} or \cite[Example 3.3]{ref20}) shows that the answer to Problem \ref{mainprob} may not always be positive in general cases. Namely, we have

	\begin{example}\label{eg2}
		Let $E$ be an $S^1$ bundle over the $K_3$ surface with non-divisible Euler class. It is not hard to show that $E$ is simply connected, thereby carrying a PSC metric due to \cite{ref5}. However, it's well known that the $K_3$ surface does not carry a PSC metric by the $\hat{A}$-genus obstruction.
	\end{example}
	
	Example \ref{eg2} demonstrates that neither the non-vanishing Rosenberg index property nor the PSC-metric obstruction property remains stable under the construction of $S^1$ bundles over certain manifolds. This fundamental difference between twisted and trivial bundles highlights the significance of Problem \ref{mainprob}. To formulate our first main result, we introduce the following class of manifolds:
	
	\begin{definition}\label{defn: NPSC+}
		Let $X$ be an oriented closed manifold. We say that $X$ is in the class $NPSC^+$, if for any oriented closed manifold $X'$ that admits a degree $1$ map $f:X'\longrightarrow X$, $X'$ admits no PSC metric.
	\end{definition}
	
	By the results in \cite{CLL}, \cite{ref6}, and \cite{ref21}, the following classes of manifolds are contained in $NPSC^+$:
	\begin{itemize}
		\item Enlargeable manifolds of dimension no greater than $7$ or with spin universal covering (see \cite[Section 5]{ref6} for the definition);
		\item Schoen-Yau-Schick (SYS) manifolds of dimension no greater than $7$ (see \cite[Section 5]{ref7} for the definition);
		\item All aspherical manifolds of dimension no greater than $5$.
	\end{itemize}
	Our first main result is as follows:

	\begin{theorem}\label{thm2}
		Let $E^n$ be a closed manifold which is a fiber bundle over $B^m$ with the fiber $F$, satisfying
		
		\begin{itemize}
			\item $B = X\#M$, where  $M$ is an arbitrary closed manifold;
			\item $X$ is a Cartan-Hadamard manifold or a $3$-dimensional $K(\pi,1)$ manifold;
			\item the homomorphism induced by the inclusion map $i_*:\pi_1(F)\longrightarrow\pi_1(E^n)$ is injective.
		\end{itemize}
		Assume that one of the following properties is satisfied:
		\begin{enumerate}
			\item $F$ is $K(\pi,1)$, $n\le 7$, $n\le m+5$;
			\item $F$ is in the class $NPSC^+$, $E$ is spin or $F$ is totally nonspin, $n\le 7$;
			\item $F$ is spin and has non-vanishing Rosenberg index, and $E$ is spin.
		\end{enumerate}
		Then $E^n$ admits no PSC metric. Moreover, in the case (1), any metric on $E^n$ with nonnegative scalar curvature is flat. In particular, the conclusion holds true if $F = S^1$.
	\end{theorem}
	
	Theorem \ref{thm2} has extended the scope of Example \ref{eg1} (a)(b)(c) above. In fact, the base space in Theorem \ref{thm2} includes all those of interest of dimension no greater than 3, since an oriented closed surface admits no PSC if and only if it is Cartan-Hadamard, while a closed 3-manifold admits no PSC if and only if it contains a $K(\pi,1)$ factor in its prime decomposition. We would like to point out that the incompressible condition here is natural since it automatically holds when the base space is aspherical. 
	
	Another point is that we are able to include a broader class of non-spin fibers when the total space has dimension no greater than 7. Over the past forty years, most results concerning topological obstruction to PSC have been obtained through Dirac operator methods, as exemplified by the works of \cite{ref6}\cite{Ros83}\cite{St92}. However, these results require certain spin conditionS on the manifold. Instead of using the Dirac operator, our major tool is based on the variational method, which is not widely used in deriving PSC topological obstruction results. This enables us to include the non-spin fibers, as well as certain spin \textit{SYS} fibers whose PSC obstruction could not be detected by any present Dirac operator method, as is shown in \cite{ref25}. The dimension 7 restriction in (1)(2) in Theorem \ref{thm2} arises in the regularity issue of minimal hypersurfaces, which is possible to be relaxed in light of recent works \cite{CMS23} and \cite{SY17}.
	
	We would like to make a brief remark here that, even in the trivial bundle (product) case, Problem \ref{mainprob} remains highly nontrivial. As discussed in §5 of \cite{ref7}, Gromov observed that the product of two enlargeable manifolds, or the product of an SYS manifold with an overtorical manifold admits no PSC metric. However, it was demonstrated in \cite{GH24} that there exist three SYS 4-manifolds whose product does admit a PSC metric (see \cite{GH24} for further examples). This indicates that, in full generality, Problem \ref{mainprob} presents substantial challenges at this stage of research.
	
	$\quad$

	We next focus on the case of $S^1$ bundles. We begin with the following general result for enlargeable manifolds.
	\begin{theorem}\label{thm3.5}
		Let $B$ be an $n$-dimensional enlargeable manifold whose universal covering is hyperspherical (see \cite{ref6}, Sec.5 for the definition) and $E$ an $S^1$ bundle over $B$ ($n+1\le 7$). If the fiber of $E$ is incompressible, then $E$ is enlargeable and carries no PSC metric.
	\end{theorem}
	
	Several remarks are in order regarding the hypotheses of the theorem. First, the condition that the universal covering being hyperspherical would not always be guaranteed by the enlargeability of $B$, as demonstrated by examples in \cite{BH09}. Second, the incompressibility assumption is essential, as compressible fibers may reduce the macroscopic dimension (in the sense of \cite{Gr96}), thereby potentially destroying the enlargeability property of the total space. However, for some more concrete base manifolds, at least for all 3-manifolds which admit no PSC metric, we can establish the following result for possibly compressible fibers:
	
	\begin{theorem}\label{thm4}
		Let $B^n$ be a closed manifold and $E^{n+1}$ an $S^1$ bundle over $B$ ($n+1\le 7$), satisfying
		\begin{itemize}
			\item $B = X\#M$, where $M$ is an arbitrary closed manifold;
			\item $X$ is a Cartan-Hadamard manifold or a $3$-dimensional $K(\pi,1)$ manifold.
		\end{itemize}
		Assume that the $S^1$ fiber of $E$ is homotopically non-trivial, then $E$ admits no PSC metric, and any metric on $E$ with nonnegative scalar curvature is flat. Moreover, the conclusion remains true for manifolds admitting a degree 1 map to $E$.
	\end{theorem}
	
	In combination with the classification of $3$-dimensional closed PSC manifolds \cite{ref6}, we conclude that the twisted $S^1$ stability holds in dimension $3$ provided the fiber is homotopically nontrivial. This stands in sharp contrast to higher-dimensional cases (see Example \ref{eg: fail in n ge 4}). As a corollary, we obtain the following class of manifolds with the twisted $S^1$ stability property:
	\begin{corollary}\label{cor2}
		Let $B$ be a closed manifold. Then $B$ has the \textit{twisted $S^1$ stability property} if one of the following holds:
		\begin{enumerate}
			\item The universal covering of $B$ is hyperspherical, and the image of the Hurewicz homomorphism in $H_2(B)$ is zeros;
			\item $B$ is a $3$-manifold that admits no PSC metric, and it contains no $S^2\times S^1$ factor in its prime decomposition.
		\end{enumerate}
	\end{corollary}
	
	In \cite{ref7}, Gromov studied codimension 2 PSC obstruction on $B\times\mathbb{R}^2$ through quadratic decay inequalities, where $B$ is an enlargeable manifold. He subsequently posed the natural question of whether these results extend to nontrivial bundles, that is, whether analogous obstruction exist for $\mathbb{R}^2$ vector bundles over $B$. Gromov's question is fundamentally connected to PSC obstruction for $S^1$ bundles over $B$, the hypersurface separating the zero section $B$ from the infinity. This appeals to suitable versions of the twisted $S^1$ stability results. Once these are established, the codimension 2 obstruction is a consequence of the dimension reduction principle in scalar curvature geometry. We have the following theorem:
	
	\begin{theorem}\label{thm9}
		Let $X^n$ ($n\le 7$) be a 2-dimensional vector bundle over $B$ where $B$ has the \textit{dominated twisted $S^1$ stability property} in the sense of Definition \ref{defn: twisted S^1 property}, and let $Y$ be a non-compact manifold which admits a degree 1 map to $X$. Then for any metric on $Y$ there exists a constant $R_0$, such that
		\begin{align}\label{1}
			\inf_{x\in B(R)}Sc(x)\le \frac{4\pi^2}{(R-R_0)^2}.
		\end{align}
	\end{theorem}
	
	Similarly, we can prove the following corollary, which gives a partial affirmative answer to Gromov's question in dimension 3. Here we state the result in a more general form when $X$ deformation retracts to a codimension 2 submanifold $B$, which includes the case of particular interest where $X$ is a vector bundle over $B$ .
	
	\begin{corollary}\label{cor3}
		Let $X^5$ be an oriented 5-manifold admitting an oriented closed 3-manifold $B^3$ (which does not admit any metric of positive scalar curvature) as a deformation retract. Suppose the boundary of a tubular neighborhood of $B^3$ in $X^5$ is an $S^1$-bundle over $B^3$ whose fiber represents a non-trivial homotopy class in $X^5 \setminus B^3$. Then $X^5$ admits no uniformly PSC metric.
	\end{corollary}
	
	For more progress of the codimension 2 obstruction for PSC, the readers may refer to \cite{CRZ} and \cite{HPS15}. We leave more details of this question to Sec.5.
	
	$\quad$
	
	We conclude this section with some remarks on the key difficulties and ideas underlying the proofs of our main theorems. For the case of the $S^1$ bundle, a natural idea is to apply the Schoen-Yau descent argument \cite{ref21} to the homology class with non-zero intersection with the fiber. However, this requires the fiber to be homologically free, which is not satisfied in most situations (See Sec.4.3 for more detailed discussions). The base space presents additional complications, being inherently more difficult to handle than fibers due to its non-embeddability into the total space. As one could see, in Example \ref{eg1}, (a)(b)(d)(e) all focus on the case of aspherical base spaces. In cases (a) and (b), the asphericity ensures a certain topological condition. In the case of (e), the advantage of the aspherical condition forces the pullback bundle on the universal covering to be trivial and one could argue by directly constructing the hyperspherical map. However, in non-aspherical cases like the connected sum $X\#M$, this argument will definitely fail where no covering space yields topologically simple pullback bundles.
	
	To show how we overcome these difficulties, we present the main ideas appearing in the proof of Theorem \ref{thm2} and Theorem \ref{thm4} respectively. For Theorem \ref{thm2}, we introduce a new conception called the \textit{incompressible depth} and obtain its estimates in several cases. This estimate reveals that the fiber bundle cannot become arbitrarily "deep" in certain directions, yielding a contradiction if one assumes the existence of a PSC metric. This gives the proof of Theorem \ref{thm2}. There are two merits of doing so. First, it preserves the crucial geometric largeness property (analogous to the enlargeability argument).  Second, it effectively eliminates concerns arising from the bundle's topological twisting.
	
	For Theorem \ref{thm4}, we will derive a contradiction by demonstrating that the $S^1$ bundle must be \textit{wide} in specific directions, then applying the width estimate inequality for \textit{SYS} bands in \cite{ref7}. This requires us to verify the \textit{SYS} condition for a certain band. To be more precise, we have to show that a certain 2-homology class lies outside the image of the Hurewicz homomorphism. However, due to the topological complexity, it can hardly be achieved via direct computation. To achieve our goal, we will distinguish spherical classes by counting their intersection numbers with a certain codimension 2 submanifold, and show that the target homology class does not satisfy the same intersection property. These elements will complete the proof of Theorem \ref{thm4}.
	
	$\quad$
	
	The rest of the paper is arranged as follows. In Sec.2 we introduce the conception of \textit{incompressible depth} and obtain its estimates. Building on the depth estimates developed in Sec.2, we prove Theorem \ref{thm2} in Sec.3. The $S^1$ bundle case will be thoroughly discussed in Sec.4, where the proof of Theorem \ref{thm3.5} and Theorem \ref{thm4} are presented. In Sec.5, we study the codimension 2 obstruction of PSC in twisted settings, and give a partial answer to Gromov's question as an application of our results.
	
	\section{$\mu$-bubble and incompressible depth in manifolds with boundary}
	The classical obstruction to the existence of incompressible hypersurfaces in PSC manifolds was first established in the seminal works of Schoen--Yau \cite{SY1} and Gromov--Lawson \cite{ref6}. In recent years, this fundamental result has been extended in various directions by numerous authors \cite{CRZ,ref2,Zei}. 
	
	Building upon these developments and utilizing the $\mu$-bubble technique, we introduce in this section the incompressible depth and establish its estimates for manifolds with boundary that have uniformly PSC. We begin by recalling some basic notions and results concerning $\mu$-bubbles from \cite{ref23}.
	
	\begin{definition}\label{def1}
		A band is a connected compact manifold $M$ together with a decomposition:
		\begin{align*}
			\partial M = \partial M^+ \sqcup \partial M^-
		\end{align*}
		where $\partial_\pm M$ are two collections of boundary components of $M$. If $M$ is endowed with a Riemannian metric, we call such a band a Riemannian band. The width of the band is defined to be the distance between $\partial M^+$ and $\partial M^-$.  
	\end{definition}
	
	The following lemma is useful in the band width estimate:
	
	\begin{lemma}(\cite{ref23}, Lemma 4.1)\label{lem1}
		Assume $(M,g)$ is a smooth Riemannian band such that width$(M,g) > 2l$, then there exists a surjective smooth map $\phi: (M,g) \longrightarrow [-l,l]$ with $Lip(\phi)<1$ such that $\phi^{-1}(-l) = \partial_-M$ and $\phi^{-1}(l) = \partial_+M$.
	\end{lemma}
	
	With the function $\phi$ we can define the $\mu$-bubble functional on an $n$-dimensional Riemannian band $M$: For any smooth function $h:(-T,T)\longrightarrow \mathbb{R}$ with $0<T<l$, we define:
	
	\begin{align}\label{21}
		\mathcal{A}^h(\Omega) = \mathcal{H}^{n-1}(\partial^*\Omega)-\int_M(\chi_{\Omega}-\chi_{\Omega_0})h\circ \phi d\mathcal{H}^n,  \quad\Omega_0 = \lbrace\phi<0\rbrace,
	\end{align}
	where $\Omega$ is any Caccippoli set in the interior of $M$ with reduced boundary $\partial^*\Omega$ such that
	\begin{align}\label{22}
		\Omega\Delta\Omega_0 \subset\subset \mathcal{D}(h\circ \phi) = \lbrace -T<\phi<T\rbrace
	\end{align}
	and $\chi_\Omega$ is the characteristic function of the region $\Omega$. 
	
	The following lemma indicates the existence and regularity of the $\mu$-bubble.
	
	\begin{lemma}(\cite{ref23}, Proposition 2.1)\label{lem2}
		Assume that $\pm$T are regular values of $\phi$ and the function $h$ satisfies
		\begin{align*}
			\lim\limits_{t\to-T}h(t) = +\infty\quad and\quad \lim\limits_{t\to T}h(t) = -\infty.
		\end{align*}
		Then there exists a smooth minimizer $\hat{\Omega}$ for $\mathcal{A}^h$, such that $\hat{\Omega}\Delta\Omega_0\subset\subset\mathcal{D}(h\circ \phi)$ provided $n\le 7$.
	\end{lemma}
	
	The following class $\mathcal{C}_{\mathrm{deg}}$ was introduced in \cite{ref2}.
	
	\begin{definition}(\cite{ref2})\label{def2}
		The class $\mathcal{C}_{\mathrm{deg}}$ consists of all closed aspherical manifolds $M$ satisfying the following property:
		\begin{itemize}
			\item If a manifold $N$ admits a continuous map $f \colon N \to M$ of non-zero degree, then $N$ does not admit a PSC metric.
		\end{itemize}
	\end{definition}
	
	Given an arbitrary class $\mathcal{C}$ of closed manifolds, we now introduce the notion of $\mathcal{C}$-incompressible depth.
	
	\begin{definition}\label{def3}
		Let $M$ be a compact Riemannian manifold with boundary.
		\begin{itemize}
			\item Define $\mathcal{S}_{\mathcal{C}}$ as the collection of all embedded, two-sided, incompressible hypersurfaces $\Sigma \subset M$ such that $\Sigma \in \mathcal{C}$.
			\item The $\mathcal{C}$-incompressible depth of $M$ is defined by
			\[
			\sup_{\Sigma \in \mathcal{S}_{\mathcal{C}}} d(\Sigma, \partial M),
			\]
			where $d(\Sigma, \partial M)$ denotes the distance from $\Sigma$ to the boundary $\partial M$.
		\end{itemize}
		In other words, this is the maximal distance achievable by an incompressible hypersurface in $\mathcal{C}$ from the boundary of $M$.
	\end{definition}

	The rest of the section is devoted to estimates of the incompressible depth in manifolds with boundary. We begin with the case of the $\mathcal{C}_{deg}$ class.
	
	\begin{proposition}\label{thm6}
		Let $M$ be a Riemannian manifold with boundary such that $\mathcal{S}_{\mathcal{C}_{deg}}\ne\varnothing$. Assume that the scalar curvature is greater than or equal to $1$, then the $\mathcal{C}_{deg}$ incompressible depth of $M$ is no greater than $\pi$.
	\end{proposition}
	
	\begin{proof}
		Assume the proposition is not true, then we can find a compact Riemannian manifold $M$ with boundary satisfying $Sc(M)\ge 1$ and a hypersurface $\Sigma$ in the class $\mathcal{S}_{\mathcal{C}_{deg}}$ in Definition \ref{def3}, such that:
		\begin{align*}
			d(\Sigma,\partial M) > l > \pi.
		\end{align*}
		Since $\pi_1(\Sigma)\longrightarrow\pi_1(M)$ is injective, we can find a covering $\tilde{M}$ of $M$ satisfying $\pi_1(\Tilde{M}) = \pi_1(\Sigma)$, while the distance between $\Sigma$ and the boundary of $\tilde{M}$ is still larger than $l$. The inclusion map $\Tilde{i}$ then induces an isomorphism $\Tilde{i}_*:\pi_1(\Sigma)\longrightarrow\pi_1(\Tilde{M})$. Since $\Sigma$ is aspherical, by \cite{ref11} (Theorem 1B.9), there exists a continuous map $j:\Tilde{M}\longrightarrow\Sigma$, such that the homomorphism $j_*:\pi_1(\Tilde{M})\longrightarrow\pi_1(\Sigma)$ equals $\Tilde{i}_*^{-1}$. Consequently, $j_*\circ i_*$ equals the identity map on $\pi_1(\Sigma)$. Once again, due to \cite{ref11} (Theorem 1B.9), we conclude that $j\circ i$ is homotopic to the identity.
		
		Next, we show that $\Sigma$ separates $\tilde{M}$ into two components. Suppose not, then $\tilde{M}\setminus \Sigma$ must be connected. Thus there exists a closed curve $\gamma$ that intersects $\Sigma$ transversally at exactly one point. Meanwhile, since the induced map $\tilde{i}_*\colon \pi_1(\Sigma) \to \pi_1(\tilde{M})$ is an isomorphism, there exists a curve $\alpha$ in $\Sigma$ such that $\tilde{i}(\alpha)$ is homotopic to $\gamma$. Noticing that $\Sigma$ is two-sided in $\tilde{M}$, we may perturb $\alpha$ slightly so that the intersection number between $\alpha$ and $\Sigma$ vanishes. However, this contradicts the homotopy invariance of intersection numbers.
		
		Let $d_0$ be a smooth approximation of the signed distance function to $\Sigma$, with the sign with respect to two connected components of $\tilde{M}\backslash \Sigma$. Let $\pi<T<l$ be a number such that $\pm T$ are both regular values of $d_0$, and we obtain that
		\begin{equation}\label{23}
			\begin{split}
				&V = \lbrace x: -T\le d_0 \le T\rbrace\\
				\partial_+V &= d_0^{-1}(T),\quad \partial_-V = d_0^{-1}(-T)
			\end{split}   
		\end{equation}
		forms a compact Riemannian band since
		\begin{align}
			V = \lbrace exp_z(tN(z)): -T\le t\le T, z\in \Sigma\rbrace.
		\end{align}
		Here $N(z)$ represents the normal vector of $\Sigma$ at $z$.
		
		We can now use the $\mu$-bubble method to derive a contradiction. Define
		\begin{align*}
			\mu = \tan(\frac{1}{2}d_0)
		\end{align*}
		on $V$. Here we note that $d_0$ serves as $\phi$ in Lemma \ref{lem2}. Consider the $\mu$-bubble functional as in (\ref{21}) and (\ref{22}):
		\begin{align*}
			\mathcal{A}(\Omega) = \mathcal{H}^{n-1}(\partial^*\Omega)-\int_V(\chi_{\Omega}-\chi_{\Omega_0})\mu d\mathcal{H}^n.
		\end{align*}
		By the existence and regularity result Lemma \ref{lem2}, there exists a separating hypersurface $W$ in $V$, such that $W$ is the reduced boundary of a stable minimizer of the functional (\ref{21}). By the first and second variation formulas, we have
		\begin{align*}
			D\mathcal{A}(\psi) = \int_{W}(H-\mu)\psi d\mathcal{H}^{n-1} = 0,
		\end{align*}
		\begin{align*}
			&D^2\mathcal{A}(\psi,\psi) \\
			&= \int_{W}(|\nabla\psi|^2 +(H^2-|A|^2-Ric(N,N)-\mu H-N(\mu))\psi^2)d\mathcal{H}^{n-1}\\
			&= \int_{W}(|\nabla\psi|^2 +(\frac{1}{2}H^2-\frac{1}{2}|A|^2+\frac{1}{2}(Sc(W)-Sc(V))-\mu H-N(\mu))\psi^2)d\mathcal{H}^{n-1}\\
			&\ge 0.
		\end{align*}
		As the first variation vanishes, we obtain that $H=\mu$ on $\Sigma$. We also have the following estimate:
		\begin{align*}
			N(\mu) = \langle N, \nabla \mu \rangle \ge -\|\nabla\mu\| \ge -\frac{1}{2}\sec^2(\frac{d_0}{2}).
		\end{align*}
		
		Consequently,
		\begin{align*}
			0 \le &D^2\mathcal{A}(\psi,\psi)\\
			\le &\int_{W}(|\nabla\psi|^2-\frac{1}{2}(|A|^2-Sc(W)+\mu^2+2N(\mu)+Sc(V))\psi^2)d\mathcal{H}^{n-1}\\
			\le &\int_{W}(|\nabla\psi|^2+\frac{1}{2}Sc(W)\psi^2)d\mathcal{H}^{n-1}.
		\end{align*}
		
		Combined with the conformal transformation argument in \cite{ref21}, we conclude that $W$ admits a PSC metric. On the other hand, since $W$ and $\Sigma$ are homologous as chains (both being cobordant to $\partial_-V$ in $V$), we have
		\begin{align*}
			\deg(j|_W) = \deg(j|_\Sigma) = \deg(j\circ i) = 1 \ne 0.
		\end{align*}
		This contradicts the fact that $W$ admits a PSC metric.
	\end{proof}
	
	We define the following classes of manifolds:
	\begin{definition}
		
		$\quad$
		
		\begin{itemize}
			\item $\mathcal{N}$ denotes the collection of all smooth closed manifolds that do not admit any Riemannian metric of positive scalar curvature (PSC).
			
			\item $\mathcal{N}_{\mathrm{nspin}}$ consists of all manifolds in $\mathcal{N}$ that are \emph{totally non-spin}, meaning their universal covers are non-spin manifolds.
			
			\item $\mathcal{R}^{\neq 0}$ comprises all spin manifolds with non-vanishing Rosenberg index:
			\begin{align*}
				\alpha(M) \neq 0 \text{ in }KO_*(\mathrm{C}^*\pi_1(M)).
		\end{align*}\end{itemize} 
	\end{definition}
	We will subsequently investigate the incompressible depth estimates for manifolds belonging to these classes.
	
	\begin{proposition}\label{pro2}
		Let $M$ be a compact $n$-dimensional spin Riemannian manifold with boundary such that $\mathcal{S}_{\mathcal{N}}\ne\varnothing$ ($n=6,7$). If the scalar curvature of $M$ is greater than or equal to $1$, then the $\mathcal{N}$ incompressible depth of $M$ is no greater than $\pi$.
	\end{proposition}
	\begin{proof}
		Let $M_0$ be the non-compact manifold obtained by pasting $M$ and $\partial M\times\lbrack 0,\infty)$ along the boundary. Assume that there is an incompressible hypersurface $\Sigma$ in $M$ with depth greater than $\pi$, then $\Sigma$ is also incompressible in $M_0$. By \cite[Proposition 4.6]{CRZ}, there is a covering $\Tilde{M_0}$ of $M_0$ with $\pi_1(\Tilde{M_0})=\pi_1(\Sigma)$, such that $\Sigma$ separates the \textit{open band} (for the definition see \cite{CRZ}) $\Tilde{M_0}$. Furthermore, any hypersurface separating $\Tilde{M_0}$ admits no PSC metric. Let $V$ be as in (\ref{23}), then any hypersurface $\Sigma'$ separating $V$ also separates $\Tilde{M_0}$. In fact, $\Sigma'$ is homologous to $\Sigma$. For any curve $\alpha$ connecting two ends in $\Tilde{M_0}$ separated by $\Sigma$, the intersection number of $\alpha$ and $\Sigma$ is non-zero, so is its intersection number with $\Sigma'$. Consequently, we conclude that $\Sigma_1$ admits no PSC metric.
		
		Following the same argument as in Proposition \ref{thm6}, we derive the desired contradiction.
	\end{proof}
	
	\begin{proposition}\label{pro4}
		Let $M$ be a compact $n$-dimensional Riemannian manifold with boundary such that $\mathcal{S}_{\mathcal{N}_{nspin}}\ne\varnothing$ ($n=6,7$). If the scalar curvature of $M$ is greater than or equal to $1$, then the $\mathcal{N}_{nspin}$ incompressible depth of $M$ is no greater than $\pi$.
	\end{proposition}
	\begin{proof}
		The proof is the same as that of the preceding proposition, with a direct application of the result in \cite{CRZ}.
	\end{proof}
	
	\begin{proposition}\label{pro3}
		Let $M$ be a compact $n$-dimensional spin Riemannian manifold with boundary such that $\mathcal{S}_{\mathcal{R}^{\ne 0}}\ne\varnothing$, with scalar curvature greater than or equal to $1$, then the $\mathcal{R}^{\ne 0}$ incompressible depth of $M$ is no greater than $\pi$.
	\end{proposition}
	\begin{proof}
		Let $M_0$ be the manifold obtained by pasting $M$ and $\partial M\times\lbrack 0,\infty)$ along the boundary. Suppose there are a metric $g$ on $M$ with $Sc(M)\ge 1$ and an incompressible hypersurface $\Sigma$ with non-vanishing Rosenberg index, such that $d(\Sigma,\partial M)>\pi$. Extend $g$ to a metric $g^+$ on $M_0$ which equals the product metric at infinity. It follows that the scalar curvature of $g^+$ is bounded from below. By \cite[Theorem 1.4]{Zei20}, we see that $\inf_{x\in M_0}Sc(g^+) = -\infty$, which leads to a contradiction.
	\end{proof}

	\section{PSC obstruction on fiber bundles: Proof of Theorem \ref{thm2}}
	
	In this section, we present the proof of Theorem \ref{thm2} which is divided into two parts. The first part establishes the nonexistence of PSC metrics on $E$, while the second part addresses the rigidity in the setting of nonnegative scalar curvature. Throughout the proof, we consistently denote by $m$ the dimension of the base space and by $n$ the dimension of the total space of the fiber bundle.
	
	\subsection{The nonexistence part}
	
	We begin with the proof of the following PSC obstruction result for bundles over the torus. The proof is mainly based on the $\mu$-bubble argument.
	
	\begin{lemma}\label{lem5}
		Let $E^n$ be a fiber bundle over $T^m$ with the fiber $F$ ($n\le 7$), where $F$ is a closed manifold in the class $NPSC^+$, then $E^n$ is also in the class $NPSC^+$.
	\end{lemma}
	\begin{proof}
		We conduct induction on $m$. If $m=0$, the conclusion follows from the definition of $NPSC^+$ class. Suppose the conclusion holds true for $m-1$. If the conclusion is not true for $m$, then there exists a Riemannian manifold $X$ such that $Sc(X)\ge 1$ due to the compactness of $X$. Let $f: X\longrightarrow E$ be a map of degree 1. Consider the natural covering map $q:T^{m-1}\times\mathbb{R}^1\longrightarrow T^m$. Let $\tilde{E}$ be the pullback bundle of $E$ under $q$, so $\phi:\Tilde{E}\longrightarrow E$ is also a covering map. This yields the following commutative diagram:
		
		\begin{align*}
			\begin{CD}
				\Tilde{E}     @>\phi>>  E\\
				@VV\Tilde{p}V        @VVpV\\
				T^{m-1}\times\mathbb{R}^1     @>q>>  T^m.
			\end{CD}
		\end{align*}
		Let $h:T^{m-1}\times \mathbb{R}^1\longrightarrow \mathbb{R}^1$ be the projection map. Define $S = \Tilde{p}^{-1}(h^{-1}(0))$, an $F$ bundle over $T^{m-1}$. By the inductive hypothesis, $S\in NPSC^+$. 
		
		Consider the following diagram:
		\begin{align*}
			\begin{CD}
				\pi_{k+1}(T^{m-1}\times\lbrace 0 \rbrace)     @>>>       \pi_k(F)     @>i_3>>  \pi_k(S) @>j_1>> \pi_k(T^{m-1}\times\lbrace 0 \rbrace)   @>>>   \pi_{k-1}(F)\\
				@VVV                   @VVidV        @VVi_1 V                                  @VVi_2 V             @VVV\\
				\pi_{k+1}(T^{m-1}\times \mathbb{R}^1)     @>>>      \pi_k(F)     @>i_4>>  \pi_k(\Tilde{E})           @>j_2>> \pi_k(T^{m-1}\times \mathbb{R}^1)   @>>>   \pi_{k-1}(F).
			\end{CD}
		\end{align*}
		By applying the five lemma, we obtain that $i_1$ is an isomorphism for all $k$ (though the typical five lemma in homological algebra is stated for modules, the diagram chase in the case of groups is essentially the same). By the Whithead Theorem, $S$ is a deformation retract of $\Tilde{E}$. This ensures the existence of the map $j:\Tilde{E}\longrightarrow S$, such that $j\circ i_1$ is homotopic to the identity.
		
		Following the argument of the proof of Proposition 5.7 in \cite{ref6}, we can pull back the covering map $\phi:\Tilde{E}\longrightarrow E$ by $f$ and obtain the following diagram:
		\begin{align*}
			\begin{CD}
				\Tilde{E}     @>\phi>>  E\\
				@AA\Tilde{f}A        @AAfA\\
				\Tilde{X}     @>\Phi>>  X.
			\end{CD}
		\end{align*}
		Here $\Phi$ is a covering map and $\Tilde{f}$ is a proper map satisfying $deg(\Tilde{f})=deg(f)=1$. We have $Sc(\Tilde{X})\ge 1$ on $\Tilde{X}$. Let
		\begin{align*}
			\Tilde{h} = h\circ\Tilde{p}\circ\Tilde{f}: \Tilde{X}\longrightarrow\mathbb{R}^1
		\end{align*}
		and assume $0$ is a regular value of $\Tilde{h}$ without loss of generality. Let $\Omega_0 = \lbrace \Tilde{h}<0\rbrace$, $\Tilde{h}^{-1}(0) = \Sigma_0$, and let $d$ be a smooth approximation of the signed distance function to $\Sigma_0$. Furthermore, we define
		\begin{align*}
			\mu = \tan(\frac{1}{2}d).
		\end{align*}
		By Lemma \ref{lem2}, there exists a stable $\mu$-bubble $\Omega$ that minimizes \eqref{21}, with $\partial\Omega = \Sigma$. By the same argument as in the proof of Proposition \ref{thm6}, $\Sigma$ carries a PSC metric. On the other hand, we consider the map
		\begin{align*}
			F = j\circ\Tilde{f}: \Tilde{X}\longrightarrow S.
		\end{align*}
		Since $\Sigma$ and $\Sigma_0$ are homologous, we have:
		\begin{align*}
			deg(F|_{\Sigma}) = deg(F|_{\Sigma_0}) = deg(\Tilde{f})=1.
		\end{align*}
		This contradicts the fact that $S\in NPSC^+$.
	\end{proof}
	
	\begin{remark}
		The argument in the proof of the above theorem also applies to the case of \textit{non-zero degrees}. However, we hope that our class $NPSC^+$ could include more manifolds like the \textit{SYS} manifolds, so we require the map to be of degree $1$ in Definition \ref{defn: NPSC+}.
	\end{remark}
	
	\begin{lemma}\label{lemM}
		Let $E^n$ be an $F$ bundle over the base $B^m = X\#M$, where $X$ is a Cartan-Hadamard manifold and $M$ is a general closed manifold. Fix a Riemannian metric on $E$. Then for any given $L_0>0$, there is a region $\mathcal{W}$ in certain Riemannian covering $\Tilde{E}$ of $E$ and an incompressible hypersurface $\partial\mathcal{W}_0$ in $\mathcal{W}$, such that $\partial\mathcal{W}_0$ is an $F$ bundle over $T^{m-1}$ satisfying
		\begin{align}\label{M}
			d(\partial\mathcal{W}_0, \partial\mathcal{W})>L_0.
		\end{align}
	\end{lemma}
	\begin{proof}		
		Let $k:B\longrightarrow X$ be a smooth map which collapses $M$ to a point. Fix metrics on $X$ and $B$ such that $Lip(k)\le 1$. Denote the universal Riemannian covering of $X$ by $\Tilde{X}$, which induces the following commutative diagram:
		\begin{align}\label{31}
			\begin{CD}
				\Tilde{B}     @>\pi>>  B\\
				@VV\Bar{k}V        @VVkV\\
				\Tilde{X}     @>>>  X.
			\end{CD}
		\end{align}
		Clearly, we have
		\begin{align}\label{32}
			\mathrm{Lip}(\Bar{k}) = \mathrm{Lip}(k) \le 1.
		\end{align}
		The covering map $\pi:\Tilde{B}\longrightarrow B$ also induces a pullback bundle $\tilde{E} = \pi^*E$:
		\begin{align}\label{33}
			\begin{CD}
				\Tilde{E}     @>\Tilde{\pi}>>  E\\
				@VVqV        @VVpV\\
				\Tilde{B}     @>\pi>>  B.
			\end{CD}
		\end{align}
		satisfying
		\begin{align}\label{34}
			\mathrm{Lip}(q) = \mathrm{Lip}(p) < \infty.
		\end{align}
		It follows that $\Tilde{E}$ is a covering space of $E$ with an incompressible fiber $F$. To be more precise, if there is an element in $\pi_1(F)$ homotopic to zero in $\Tilde{E}$, then under the covering map $\Tilde{\pi}$, this yields an element in $\pi_1(F)$ homotopic to zero in $E$, which contradicts the incompressibility of $F$ in $E$.
		
		We proceed with the depth analysis. First, we identify a region $\mathcal{V}$ in $\tilde{X}$ such that $\mathcal{V} \cong T^{m-1} \times [0,1]$, satisfying
		\begin{equation}\label{35}
			\begin{aligned}
				&\partial\mathcal{V} = T^{m-1} \times \{0\} \cup T^{m-1} \times \{1\}, \\
				&\partial_0\mathcal{V} = T^{m-1} \times \left\{\frac{1}{2}\right\}, \\
				&d(\partial\mathcal{V}, \partial_0\mathcal{V}) > L
			\end{aligned}
		\end{equation}
		for arbitrarily large $L$.
		
		Since $\tilde{X}$ is a simply connected Cartan-Hadamard manifold, there exists a map $f$ that diffeomorphically maps the geodesic ball $B_R(0)$ to the unit disk $D^n$ endowed with a flat metric, where $\mathrm{Lip}(f) \leq \frac{1}{R}$. (Although the metric on $\tilde{X}$ may not have non-positive sectional curvature, it is $C^0$-equivalent to such a metric.)
		
		Fix a two-sided embedding of $T^{m-1}$ in $D^m$ and select its tubular neighborhood $V \cong T^{m-1} \times [0,1]$, such that
		\[
		d(\partial V, \partial_0 V) > \delta
		\] 
		holds in $D^n$, where $\delta$ is a small constant depending only on the dimension $n$. Then, $\mathcal{V} = f^{-1}(V)$ satisfies the desired properties provided $R$ is chosen sufficiently large.
		
		We next perform the construction in $\tilde{B}$. The construction in (\ref{31}) shows that $\tilde{B}$ is the connected sum of $\tilde{X}$ and infinitely many copies of $M$:
		\begin{align*}
			\tilde{B} = \tilde{X}\mathbin{\#_{p_1}}M\mathbin{\#_{p_2}}M\mathbin{\#}\cdots
		\end{align*}
		where $\#_{p_i}$ denotes taking the connected sum in a small neighborhood of $p_i\in\tilde{X}$. The map $\overline{k}\colon \tilde{B}\to \tilde{X}$ in (\ref{31}) collapses these $M$ components to the points $p_1,p_2,p_3,\dots$.
		
		By making small perturbations of the hypersurfaces $\partial\mathcal{V}$ and $\partial_0\mathcal{V}$, we may assume none of the $p_i$ lie on $\partial\mathcal{V}$ or $\partial_0\mathcal{V}$. Define
		\begin{align*}
			\mathcal{U} &= \overline{k}^{-1}(\mathcal{V})
		\end{align*}
		and set
		\begin{align*}
			\partial\mathcal{U} &= \overline{k}^{-1}(\partial\mathcal{V}), \\
			\partial_0\mathcal{U} &= \overline{k}^{-1}(\partial_0\mathcal{V}).
		\end{align*}
		
		Since $\partial\mathcal{V}$ and $\partial_0\mathcal{V}$ do not intersect any $p_i$, the maps $\partial\mathcal{U}\to \partial\mathcal{V}$ and $\partial_0\mathcal{U}\to \partial_0\mathcal{V}$ induced by $\overline{k}$ are diffeomorphisms. This ensures that the components of $\partial\mathcal{U}$ and $\partial_0\mathcal{U}$ are tori. Topologically, $\mathcal{U}$ is the connected sum of $T^{m-1}\times[0,1]$ with finitely many copies of $M$.  
		
		We need to verify the incompressible condition and depth condition of $\mathcal{U}$. First, recall the basic fact that for manifolds $M_1,M_2$ of dimension at least 3, $\pi_1(M_1\#M_2) = \pi_1(M_1)*\pi_1(M_2)$ by Van Kampen's theorem. Therefore, the inclusion map $M_1\backslash D^n \hookrightarrow M_1\#M_2$ always induces an injection on fundamental groups. This implies that
		\begin{equation}\label{36}
			\begin{split}
				&\pi_1(T^{m-1}\times[0,1] - \{q_1,q_2,\dots,q_r\})\\
				\longrightarrow &\pi_1(T^{m-1}\#_{q_1}M\#_{q_2}M\#\cdots\#_{q_r}M)\cong\pi_1(\mathcal{U})
			\end{split} 
		\end{equation}
		is injective, where the notation $\#_{q_i}$ means performing the connected sum in a small neighborhood of the point $q_i$. Since the middle torus $T^{m-1}\times\{\frac{1}{2}\}$ is incompressible in $T^{m-1}\times[0,1]$, it follows from (\ref{36}) that the middle torus $\partial_0\mathcal{U}$ is incompressible in $\mathcal{U}$.
		
		Since $Lip(\bar{k})\le 1$ from (\ref{32}), we also have
		\begin{align*}
			d(\partial\mathcal{U},\partial_0\mathcal{U}) > L
		\end{align*} 
		by \eqref{35}.
		
		Now we can construct the desired region $\mathcal{W}$ in $\Tilde{E}$. Define $\mathcal{W}$ to be the restriction of $\Tilde{E}$ on $\mathcal{U}$. Since $Lip(q) = Lip(p) = C<+\infty$ by (\ref{34}), we have
		\begin{align*}
			d(\partial\mathcal{W},\partial_0\mathcal{W}) > \frac{L}{C}
		\end{align*}
		where $\partial_0\mathcal{W}$ is the restriction of $\mathcal{W}$ on $\partial_0\mathcal{U}$. 
		
		For the remaining of the proof, it suffices to show $\partial_0\mathcal{W}$ is an incompressible hypersurface in $\mathcal{W}$. Consider the following commutative diagram:
		\begin{align}\label{chase}
			\begin{CD}
				\pi_1(F)     @>i_1>>  \pi_1(\partial_0\mathcal{W}) @>j_1>> \pi_1(\partial_0\mathcal{U})\\
				@VVidV        @VV\varphi V                                  @VV\psi V\\
				\pi_1(F)     @>i_2>>  \pi_1(\mathcal{W})           @>j_2>> \pi_1(\mathcal{U}).
			\end{CD}
		\end{align}
		All the maps of vertical arrows are induced by inclusion maps. Recall $\psi:\partial_0\mathcal{U}\longrightarrow \mathcal{U}$ is injective. We now show that $\varphi$ is injective. Take $a \in \pi_1(\partial_0\mathcal{W})$. If $j_1(a)\ne 0$, then $\psi\circ j_1(a)\ne 0$ due to the injectivity of $\psi$. This indicates that $\varphi(a)\ne 0$ by the commutativity of the diagram. On the other hand, if $j_1(a) = 0$, then the exactness yields an element $b\in\pi_1(F)$ with $a = i_1(b)$. Since the inclusion $\pi_1(F)\longrightarrow\pi_1(\Tilde{E})$ is injective, it follows that the inclusion $\pi_1(F)\longrightarrow\pi_1(\mathcal{W})$ is also injective, which shows $i_2$ is injective. Therefore, $\varphi\circ i_1(b) = i_2(b) = 0$ indicates $b = 0$. Consequently, $\varphi$ is injective. Since $L$ can be arbitrarily large, we complete the proof of the lemma.
	\end{proof}

	\begin{proposition}\label{proM1}
		Let $E$ be as in Theorem \ref{thm2}, then $E$ admits no PSC metric if $X$ is Cartan-Hadamard.
	\end{proposition}
	\begin{proof}
		We argue by contradiction. By the compactness of $E$, we may assume there exists a metric on $B$ with $Sc(E) \geq 1$. Choose $\mathcal{W}$ and $\partial_0\mathcal{W}$ as in Lemma \ref{lemM}. We proceed case by case.
		
		\textbf{Case 1:} $F$ is $K(\pi,1)$ and $n \leq 7$. By Lemma \ref{lem5} and \cite{CLL}, $\partial_0\mathcal{W}$ belongs to class $\mathcal{C}_{deg}$ in Definition \ref{def3}. In the terminology of Section 2, this means $\mathcal{W}$ has $\mathcal{C}_{deg}$ incompressible depth at least $L/C$. Since $\mathcal{W}$ has scalar curvature $\geq 1$ and $L_0$ in (\ref{M}) can be taken arbitrarily large, we obtain a contradiction from the $\mathcal{C}_{deg}$ incompressible depth estimate in Proposition \ref{thm6}.
		
		\textbf{Case 2:} $F$ is spin with non-vanishing Rosenberg index, and $E$ is spin. From \cite[Theorem 1.5]{Zei}, as an $F$-bundle over $T^{m-1}$, $\partial_0\mathcal{W}$ also has non-vanishing Rosenberg index. The contradiction then follows from the $\mathcal{R}^{\neq 0}$ incompressible depth estimate in Proposition \ref{pro3}.
		
		\textbf{Case 3:} $F$ is in class $NPSC^+$, with either $E$ spin or $F$ totally non-spin, and $n \leq 7$. We assume $m \geq 2$ (the $m=1$ case follows from Lemma \ref{lem5}).
		
		\begin{itemize}
			\item[(a)] $E$ is spin. By Lemma \ref{lem5}, $\partial_0\mathcal{W}$ admits no PSC metric. For $n \geq 6$, the contradiction follows from the $\mathcal{N}$ incompressible depth estimate in Proposition \ref{pro2}. For $n \leq 5$, since closed manifolds of dimension $\leq 3$ with no PSC metric have non-vanishing Rosenberg index \cite{HS06}, the conclusion follows from Case 2.
			
			\item[(b)] $F$ is totally non-spin. This requires $n-m \geq 4$. With $m \geq 2$, we have $n \geq 6$. Considering the pullback bundle of the covering map $\mathbb{R}^{m-1} \to T^{m-1}$, we see $\partial_0\mathcal{W}$ is also totally non-spin (note the normal bundle of a single fiber is trivial). The contradiction then follows from Lemma \ref{lem5} and Proposition \ref{pro4}.
		\end{itemize}
	\end{proof}
	
	\begin{proposition}\label{proM2}
		Let $E$ be as in Theorem \ref{thm2}, then $E$ admits no PSC metric if $m=3$ and $X$ is $K(\pi,1)$.
	\end{proposition}
	\begin{proof}
		To begin, let's recall the hyperbolisation conjecture, a corollary of the geometrization conjecture (see \cite{MG}), which says a 3-dimensional closed aspherical manifold either is hyperbolic or contains a $\mathbb{Z}\oplus\mathbb{Z}$ subgroup in its fundamental group (see \cite[12.9.4]{GT}). Building on this we can assume $\pi_1(X)$ contains a subgroup isomorphic to $\mathbb{Z}\oplus\mathbb{Z}$, since the hyperbolic case is contained in the Cartan-Hadamard case. It follows that $X$ contains an incompressible torus, and the same thing holds for $B=X\#M$. By a similar diagram chase as (\ref{chase}), $E$ contains a nice incompressible hypersurface. It follows from \cite[Theorem 1.5]{CRZ}, \cite[Theorem 1.1]{ref2}, \cite[Theorem 1.7]{Zei} and the same argument in Proposition \ref{proM1} that $E$ admits no PSC metric.
	\end{proof}
	
	\subsection{The rigidity part}
	In this subsection, we deal with the rigidity part of Theorem \ref{thm2}. If there is a metric $g$ on $E$ with nonnegative scalar curvature, it follows from \cite{Kaz82} that $g$ is Ricci flat. To handle the Ricci flat case, we will analyze the volume growth of a certain covering of the bundle and make use of the structure of compact Ricci flat manifolds.
	
	\begin{definition}\label{def4}
		For a non-compact Riemannian manifold $(M,g)$, define the volume growth rate
		\begin{align*}
			r(M,g) = \sup\{\alpha\in\mathbb{R}_+, \liminf_{R\to\infty}\frac{\mathrm{Vol}(B(p,R))}{R^\alpha}>0\}
		\end{align*}
		where $B(p,R)$ is the geodesic ball centered at $p$ with radius $R$. Furthermore, we define
		\begin{align*}
			r(M) = sup\lbrace\alpha\in\mathcal{A}\rbrace.
		\end{align*}
	\end{definition}
	It is straightforward to see $r(M)$ is independent of the choice of the basepoint $p$.
	
	\begin{example}\label{lem6}
		Let $M$ be the Riemannian product of a compact manifold $N$ and the Euclidean space $\mathbb{R}^k$, then $M$ has volume growth rate $k$. For an $n$-dimensional simply connected Cartan-Hadamard manifold $X$, the volume growth rate of $X$ is at least $n$ by the volume comparison.
	\end{example}
	
	Recall that two metrics $g_1$ and $g_2$ are equivalent if there exists a constant $C$ such that:
	\begin{align*}
		\frac{1}{C}g_1\le g_2\le Cg_1.
	\end{align*}
	It is clear that equivalent metrics have the same volume growth rate. Therefore, we have:
	\begin{lemma}\label{lem7}
		Let $\Tilde{M}$ be a covering space of a closed manifold $M$. Then any two metrics on $\Tilde{M}$ induced from $M$ have the same volume growth rate.
	\end{lemma}
	
	The following lemma establishes a fundamental relationship between the volume growth rates of a manifold and its covering space:
	\begin{lemma}\label{lem8}
		Let $p:\Tilde{M}\longrightarrow M$ be a Riemannian covering map between two non-compact Riemannian manifolds, then $r(\Tilde{M})\ge r(M)$.
	\end{lemma}
	\begin{proof}
		Let $x$, $\Tilde{x}$ be points in $M$ and $\Tilde{M}$ with $p(\Tilde{x})=x$. Let $B(x,R)$ and $B(\Tilde{x},R)$ be geodesic balls of radius $R$ centered at $x$ and $\Tilde{x}$. Since $p:B(\Tilde{x},R)\longrightarrow B(x,R)$ is surjective and $p$ is a local isometry, the Area Formula yields that
		\begin{align*}
			\mathrm{Vol}(B(\Tilde{x},R)=\int_{B(x,R)}n(y)d\sigma
			\ge\int_{B(x,R)}1d\sigma = \mathrm{Vol}(B(x,R))
		\end{align*}
		where $n(y)$ denotes the number of the inverse image of $y$ under $p$ in $B(\Tilde{x},R)$, which completes the proof.
	\end{proof}
	
	Now we give the proof of the rigidity part of Theorem \ref{thm2}.
	\begin{proof}[\textbf{Proof of Theorem \ref{thm2} (Rigidity part)}]
		We show that any Ricci-flat metric on $E$ must be flat when $F$ is aspherical. Assume, for contradiction, that $E$ admits a non-flat Ricci-flat metric. 
		
		By the structure theorem for compact Ricci-flat manifolds \cite{FW}, there exists a finite covering $p: W \times T^l \to E$, where $W$ is a closed simply connected Ricci-flat manifold. The non-flatness of $E$ implies that $W$ is nontrivial, and thus $\dim W \geq 4$ (since no closed simply connected Ricci-flat manifolds exist in dimensions $\leq 3$). 
		
		Pulling back this covering to $\tilde{E}$, we obtain the following commutative diagram: 
		\begin{align*}
			\begin{CD}
				\Tilde{E}     @>\pi>>  E\\
				@AAqA        @AApA\\
				Z     @>\Tilde{\pi}>>  W\times T^l.
			\end{CD}
		\end{align*}
		where all the maps are covering maps. Since $W$ is simply connected, $Z$ must be of the form $W\times T^{l-s}\times \mathbb{R}^s$. If $X$ is Cartan-Hadamard, by Lemma \ref{lem8} and Lemma \ref{lem10}, we have
		\begin{align*}
			m\le r(\Tilde{E})\le r(Z) = s.
		\end{align*}
		Therefore,
		\begin{align*}
			l-s\le n-dimW-s\le n-m-1.
		\end{align*}
		If $X$ is of dimension 3, we have
		\begin{align*}
			l-s\le n-dimW \le n-4 = n-m-1.
		\end{align*}
		Since $p$ induces a finite covering, the same thing also holds for $q$. Therefore, $\pi_1(Z) =\mathbb{Z}^{l-s}$ is a finite index subgroup of $\pi_1(\Tilde{E})$. Denote the embedded image of $\pi_1(F)$ in $\pi_1(\tilde{E})$  by $H$. By Lemma \ref{A3}, $H\cap \pi_1(Z)$ is a finite index subgroup of $H$. Let $\mathbb{Z}^t = H\cap \pi_1(Z)$, then we have $t\le l-s \le n-m-1$. On the other hand, $H$ is the fundamental group of a closed aspherical manifold $F$. It follows that $F$ has a finite covering $\Tilde{F}$ with the fundamental group $\mathbb{Z}^t$, where $\Tilde{F}$ is also aspherical. Hence, $\Tilde{F}$ is homotopic equivalent to $T^t$, so $H_{n-m}(\tilde{F},\mathbb{Z}_2) = 0$. This leads to a contradiction as $\tilde{F}$ is of dimension $n-m$. 
	\end{proof}
	
	\section{PSC obstruction on $S^1$ bundles}
	We study the PSC obstruction on $S^1$ bundles in this section. By possibly passing to covering spaces, we may assume both $B$ and $E$ are orientable.
	\subsection{$S^1$ bundles over enlargeable manifolds with incompressible fibers}
	
	In this subsection, we present a proof of Theorem \ref{thm3.5}.
	\begin{lemma}\label{lem15}
		Let $M$ be a simply connected manifold and $E$ an $S^1$ principal bundle over $B$. If the fiber of $E$ is incompressible, then $E$ is trivial.
	\end{lemma}
	\begin{proof}
		Consider the long exact sequence
		\begin{align*}
			\pi_2(B)\stackrel{\partial}{\longrightarrow}\pi_1(S^1)\stackrel{i}{\longrightarrow}\pi_1(E)\longrightarrow\pi_1(M).
		\end{align*}
		It follows from the incompressible condition that $\partial=0$. For any $\sigma\in\pi_2(B)$, by Lemma \ref{A1}, we have $\partial(\sigma) = e(\mathrm{hur}(\sigma))=0$. Since $\pi_1(B) = 0$, by Hurewicz Theorem we have $e(a) = 0$ for any $a\in H_2(B)$. This concludes $e=0$ and $E$ is trivial.
	\end{proof}
	\begin{proof}[\textbf{Proof of Theorem \ref{thm3.5}}]
		Fix two metrics $g_0$, $g$ on $B$ and $E$ respectively. Let $\Tilde{g}_0$ be the pullback metric on $\Tilde{B}$, the universal covering of $B$, and let $q:(\Tilde{E},\Tilde{g})\longrightarrow (\tilde{B},\tilde{g}_0)$ be the pullback bundle under the covering map. By Lemma \ref{lem15} $\Tilde{E}$ is trivial, so we can write $\Tilde{E} = \Tilde{B}\times S^1$. We have
		\begin{align*}
			&f_1: \tilde{E}\stackrel{q}{\longrightarrow} \tilde{B}\longrightarrow S^n,\\
			&f_2:\tilde{E}\longrightarrow S^1,
		\end{align*}
		where the second map in the first row is given by the hyperspherical metric $\tilde{g}_0$, so we may assume $f_1$ is constant at the infinity with $\mathrm{Lip} f_1<\frac{1}{2}\epsilon$ and $(f_1)_{!}(pt)\ne 0\in H_1(\tilde{E})$. Let $z_k:S^1\longrightarrow S^1$ be the $k$-fold covering map. By selecting
		\begin{align*}
			k>2\epsilon^{-1}\mathrm{Lip} (f_2|_{\mathrm{supp} f_1})
		\end{align*}
		we obtain that the map $f_3 = z_k\circ f_2$ satisfies $\mathrm{Lip} f_3<\frac{1}{2}\epsilon$. It follows that
		\begin{align*}
			F :\tilde{E}\stackrel{f_1\times f_2}{\longrightarrow}S^{n-1}\times S^1\longrightarrow S^{n-1}\wedge S^1= S^n
		\end{align*}
		is a non-zero degree map with $\mathrm{Lip} F<\epsilon$ that equals the constant at the infinity, which completes the proof.
	\end{proof}
	
	\begin{remark}
		Since $\mathrm{Diff}^+(S^1)\simeq SO(2)$, we have $B\mathrm{Diff}^+(S^1)\simeq BSO(2)$. This implies that every fiberwise orientable $S^1$ bundle is induced by an orientable $\mathbb{R}^2$-vector bundle. In other word, every fiberwise orientable $S^1$ bundle admits an $S^1$-principle bundle structure.  For the remainder of this paper, we will not distinguish between these two notions.
	\end{remark}

	\subsection{$S^1$ bundles and \textit{SYS} bands}
	In this subsection, we focus on the proof of Theorem \ref{thm4}. The tools are based on Gromov's width estimate for \textit{SYS} bands in \cite{ref7}.
	
	\begin{definition}\label{def5}
		An $n$-dimensional Riemannian band $M$ is called \textit{SYS}, if there exist $\alpha_1,\alpha_2,\dots,\alpha_{n-3}\in H^1(M)$ and $\beta\in H^1(M,\partial M)$, such that
		\begin{align}\label{401}
			\lbrack M,\partial M\rbrack \smallfrown\alpha_1\smallfrown\dots\smallfrown \alpha_{n-3}\smallfrown\beta
		\end{align}
		is non-spherical, where $\lbrack M,\partial M\rbrack$ represents the fundamental class of a compact manifold with boundary.
	\end{definition}
	
	With the application of the torical symmetrization technique, Gromov proved a $4\pi$ inequality for \textit{SYS} bands in \cite{ref7}.  
	
	\begin{theorem}(\cite{ref7})\label{thm7}
		If a \textit{SYS} band $M$ has $Sc(M)\ge 1$, then $width(M)\le 4\pi$.
	\end{theorem}
	
	By combining the torical symmetrization argument with the warped $\mu$-bubble method in dimension 3, one can actually obtain the optimal upper bound $2\pi\sqrt{\frac{n-1}{n}}$ in the preceding theorem.
	
	To prove Theorem \ref{thm4}, we need some more subtle topological description for $S^1$ bundles over the connected sum. We begin with the following structure lemma, which implies that an $S^1$ bundle over a connected sum has a specific generalized connected sum description in the sense of \cite{ref2}.
	\begin{lemma}(The Structure Lemma)\label{lem11}
		Let $M$ and $N$ be $n$-dimensional closed manifolds and $E$ an $S^1$ bundle over $B = M\#N$ ($n\ge 3$). Let $S=S^{n-1}$ be the connecting sphere of $M$ and $N$ (the boundary of $M\backslash D^n$ used to form the connected sum). Then the restriction of $E$ on $S$ is trivial, and $E$ is a fiber connected sum of $E_M$ and $E_N$, where $E_M$ and $E_N$ are $S^1$ bundles over $M$ and $N$. Here the fiber connected sum of $E_M$ and $E_N$ is defined by
		\begin{align*}
			E_M|_{M\backslash D^n}\cup_{\varphi}E_N|_{N\backslash D^n}
		\end{align*}
		with $\varphi: E_M|_{\partial (M\backslash D^n)}\longrightarrow E_N|_{\partial (N\backslash D^n)}$ an $S^1$ bundle isomorphism. The notation $E_M|_A$ represents the restriction of $E_M$ on a subset $A\subset M$.
	\end{lemma}
	\begin{proof}
		We first show that the restriction of $E$ on $S$ is trivial. Recall the isomorphism class of oriented $S^1$ bundles over a certain topological space $Y$ is in 1-1 correspondence to $H^2(Y)$. Since $i^*:H^2(M)\longrightarrow H^2(M\backslash D^n)$ is an isomorphism when $n\ge 3$, the restriction of $E$ on $M\backslash D^n$ is the restriction of a certain $S^1$ bundle $E'$ on $M$. As $D^n$ is contractible, the restriction of $E'$ on $S$ is trivial, and this gives the desired result.
		
		Next, we can construct $E_M$ by filling the boundary of $E|_{M\backslash D^n}$ by the trivial bundle over $D^n$ and construct $E_N$ in a similar fashion, which gives the desired fiber connected sum decomposition.
	\end{proof}
	
	We now begin the proof of Theorem \ref{thm4}. Let us briefly outline the strategy. In the paragraph preceding Lemma \ref{lem3-2}, we construct a band in a certain covering space of $E$ and provide an explicit description of this band. The subsequent paragraph will demonstrate that this band is \textit{SYS}, with Lemma \ref{lem3-2} playing a central role in the argument. Lemmas \ref{lem3-3}, \ref{lem3-4}, and \ref{lem3-5} serve as intermediate steps in the proof of Lemma \ref{lem3-2}. With these preparations, we can then complete the proof of Theorem \ref{thm4}. In our presentation, we will mainly focus on the case that $X$ is a 3-dimensional aspherical manifold. The method for the Cartan-Hadamard case is similar.

	We start with a 3-dimensional aspherical manifold $X$, and consider its oriented covering space $\Tilde{X}$ with the fundamental group $\mathbb{Z}$. Select a curve $\gamma\in\Tilde{X}$ representing the generator of $\pi_1(\Tilde{X})$. By possibly passing to a two-fold covering we may assume $\tilde{X}$ is orientable. By the Whitehead Theorem, we see that $\Tilde{X}$ possesses $\gamma$ as a deformation retract. Let $V$ be a small tubular neighborhood of $\gamma$, then $\Sigma_0=\partial V = T^2$ due to the orientability of $\Tilde{X}$. Now we may apply Lemma \ref{A5} to obtain
	\begin{align}\label{j}
		j:\Tilde{X}\backslash V\longrightarrow \Sigma_0
	\end{align}
	such that $j|_{\Sigma_0}$ is homotopic to the identity.
	
	Define
	\begin{align}\label{404}
		\mathcal{V} = U_{R}(\Sigma_0) = \lbrace x: d(x,\Sigma_0)\le R, x\in \Tilde{X}\backslash V\rbrace
	\end{align}
	which is an overtorical band of width $R$ due to the map $j$ in (\ref{j}). The covering $\Tilde{X}\longrightarrow X$ gives rise to a covering $\Tilde{B}\longrightarrow B$ as follows:
	\begin{align}\label{31'}
		\begin{CD}
			\Tilde{B}     @>\pi>>  B\\
			@VV\Bar{k}V        @VVkV\\
			\Tilde{X}     @>>>  X.
		\end{CD}
	\end{align}
	Then we have
	\begin{align}\label{40}
		\Tilde{B} = \Tilde{X}\#_{p_1}M\#_{p_2}M\#\dots
	\end{align}
	Denote the pullback bundle under the covering map $\Tilde{B}\longrightarrow B$ by $\Tilde{E}$, and let $S$ be the sphere connecting $X$ and $M$ in $X\#M$. The structural Lemma \ref{lem11} guarantees that the restriction of $E|_S$ is trivial. Let
	\begin{align*}
		\pi^{-1}(S) = S_1\cup S_2\cup S_3\cup\dots.
	\end{align*}
	It follows that $\tilde{E}|_{S_i}$ is trivial for $i = 1,2,3,\dots$. We substitute the bundles over infinite copies of $M$ in $\Tilde{E}$ by infinitely many $D^3\times S^1$. This gives a trivial $S^1$ bundle over $\Tilde{X}$, since $H^2(\Tilde{X})=H^2(S^1) = 0$. Topologically, we see that $\Tilde{E}$ is the fiber connected sum of $\Tilde{X}\times S^1$ with infinite copies of $S^1$ bundles over $M$:
	\begin{align}\label{42}
		\Tilde{E} = \Tilde{X}\times S^1\#_{S^1}Z\#_{S^1}Z\#\dots.
	\end{align}
	Here $Z$ denotes the corresponding $S^1$ bundle over $M$.
	
	\begin{figure}
		\centering
		\includegraphics[width = 12cm]{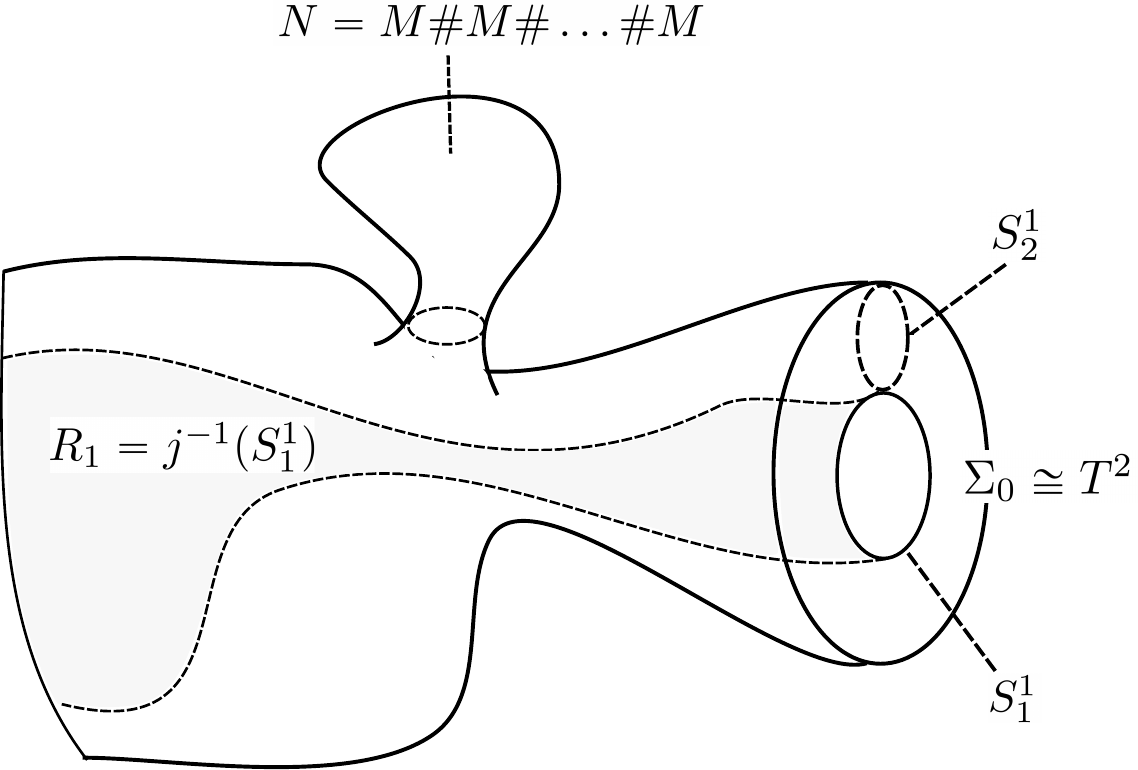}
		\caption{A schematic diagram of $\mathcal{U}$}
		\label{figure1}
	\end{figure}
	
	Next, we construct $\mathcal{U} = \Bar{k}^{-1}(\mathcal{V})$. By perturbing $\partial\mathcal{V}$ slightly, we may assume that none of the $p_i$ in (\ref{40}) lies in $\partial\mathcal{V}$. Consequently, we obtain
	\begin{align}\label{402}
		&\mathcal{U} = \mathcal{V}\#_{p_1}M\#_{p_2}M\#\dots\#_{p_k}M = \mathcal{V}\#N.
	\end{align}
	Let $\mathcal{W} = \Tilde{E}|_{\mathcal{U}}$, then
	\begin{align}\label{403}
		&\mathcal{W} = P\#_{p_1\times S^1}Z\#_{p_2\times S^1}Z\#\dots\#_{p_k\times S^1}Z = P\#_{S^1}Y
	\end{align}
	where $P$ is the $S^1$ bundle over $\mathcal{V}$, $Z$ and $Y$ are certain $S^1$ bundles over $M$ and $N$. Since $P$ is induced from the trivial $S^1$ bundle over $\Tilde{X}$, we have $P = \mathcal{V}\times S^1$. The notation $\#_{p_i\times S^1}$ here means taking the fiber connected sum at $p_i\times S^1\subset \mathcal{V}\times S^1 = P$.

	Let $\Sigma_0 \cong T^2 = S^1_1 \times S^1_2$, and denote by $R_1 = j^{-1}(S^1_1)$ and $R_2 = j^{-1}(S^1_2)$ the corresponding surfaces in $\mathcal{V}$. These surfaces can also be viewed in $\mathcal{U}$ and, under suitable arrangement (away from $N$), in $\mathcal{W}$ as well, since the restriction of $\mathcal{W}$ to $\mathcal{V}$ is trivial. A schematic diagram of $\mathcal{U}$ is presented in Figure \ref{figure1}, where the map $j$ retracts the manifold onto the right-hand side torus $\Sigma_0\cong T^2$.
	
	Let $T = S^1_2 \times S^1 \subset \Sigma_0 \times S^1 \subset \partial\mathcal{W}$, and let $R_2^+$ denote the restriction of the bundle $\mathcal{W}$ to $R_2$. Then, the following holds in the sense homology:
	\begin{align}\label{411}
		T = R_2^+ \cap (\Sigma_0 \times S^1),
	\end{align}
	and the intersection number of $T$ with $R_1$ is equal to 1. Our next goal is to prove the following:
	
	\begin{lemma}\label{lem3-2}
		Suppose the fiber of $\mathcal{W}$ has order $k$ in $\pi_1(\mathcal{W})$. Let $\sigma$ be an immersed sphere in $\mathcal{W}$, then the intersection number of $\sigma$ and $R_1$ is a multiple of $k$.
	\end{lemma}
	
	We write $\mathcal{W} = (\mathcal{V}-D^3)\times S^1\cup (Y-D^3\times S^1) = A_1\cup A_2$, and $L = S^2\times S^1$, the common boundary of $A_1$ and $A_2$. Suppose $f:S^2\longrightarrow\mathcal{W}$ represents a spherical class $\lbrack\sigma\rbrack$, and $f$ intersects transversally with $L$. Then 
	\begin{align*}
		f^{-1}(L) = \gamma_1\cup\gamma_2\cup\dots\cup\gamma_s
	\end{align*}
	which is the union of closed curves. By abuse of notation, we continue to denote the image of the curves $\gamma_i$ under $f$ by $\gamma_i$. Since $S^2$ is simply connected, each $\gamma_i$ must be contractible in $\mathcal{W}$ and thus represents a multiple of $k$ in $\pi_1(L) \cong \pi_1(S^1 \times S^2) \cong \mathbb{Z}$.
	
	These curves $\gamma_i$ partition the sphere $\sigma = S^2$ into several regions. Under the map $f$, some of these regions are mapped to $A_1$, while others are mapped to $A_2$. To compute the intersection with $R_1$ (which is a subset of $A_1$), it suffices to consider only the regions mapped to $A_1$.
	
	\begin{lemma}\label{lem3-3}(The Decomposition Lemma)
		Let $U\subset S^2$ be a region mapped to $A_1$ by $f$ as above, then the homology class in $f_*([U])\in H_2(A_1,\partial A_1)$ can be represented by the sum of some cylinders and at most one sphere.
	\end{lemma}
	
	\begin{figure}
		\centering
		\includegraphics[width = 15cm]{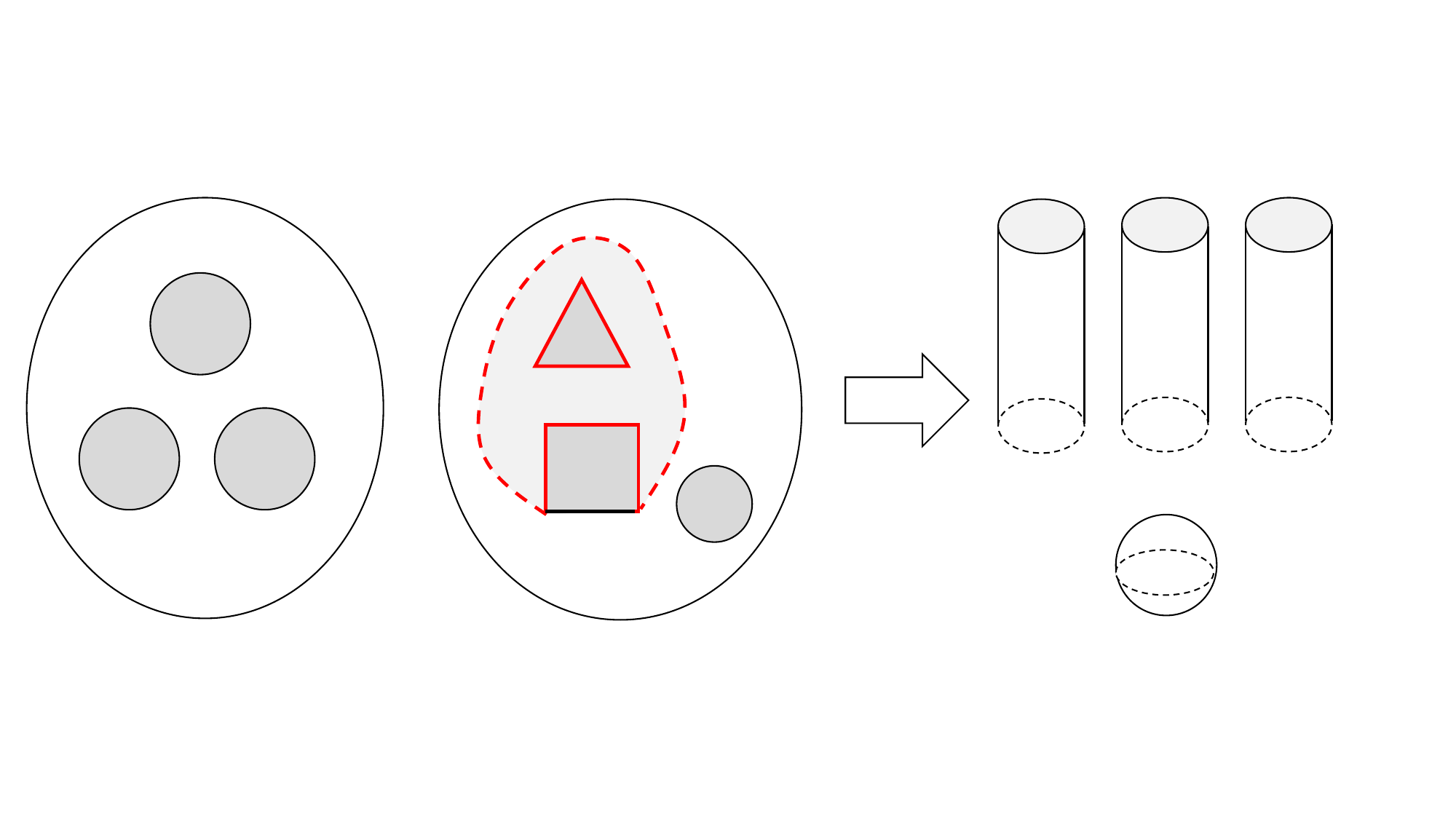}
		\caption{Decomposing the sphere into cylinders}
		\label{figure2}
	\end{figure}
	
	\begin{proof}
		Since $U$ is a region in $S^2$, it must be of the shape in Figure \ref{figure2}. We may assume that its boundary $\partial U$ is given by $\gamma_1 \cup \gamma_2 \cup \dots \cup \gamma_r$, where each $\gamma_i$ is mapped to $L$ under $f$ and represents $k$ times the generator of $\pi_1(L)$, as is shown earlier.
		
		Suppose $\gamma_i$ represents $k a_i$ times the generator in $\pi_1(L) \cong \pi_1(S^2 \times S^1)$. Fix a standard curve $\beta = \{0\} \times S^1 \subset L$ with the base point $p$. Then, $\gamma_i$ can be homotoped to $k a_i \beta$ in $L$ without changing the homology class of the image of $U$. Thus, we may assume $\gamma_i$ is exactly $k a_i \beta$.
		
		Since the $\gamma_i$'s form the boundary of $U$, their sum in homology must vanish:
		\[
		\theta = [\gamma_1] + [\gamma_2] + \dots + [\gamma_r] = 0 \in H_1(A_1).
		\]
		Since $A_1 = (\mathcal{V} - D^3) \times S^1$ and $\gamma_i \subset L \subset \partial A_1$, the class $\theta$ must lie in the image of the map induced by the inclusion map
		\[
		i_{1*}: H_1(S^1) \to H_1((\mathcal{V} - D^3) \times S^1),
		\]
		which is injective by the Künneth formula. Hence, $\theta = 0 \in H_1(S^1)$, and consequently, $\theta = 0 \in \pi_1(S^1)$. This implies that the coefficients must satisfy $a_1 + a_2 + \dots + a_r = 0$.
		
		We proceed by induction on $r$. When $r=1$, $U$ is a disk and $\gamma_1$ is contractible in $A_1$. Since the map $\pi_1(L)\to \pi_1(A_1)$ is injective, $\gamma_1$ is also contractible in $L$. We can then complete $U$ into a sphere by adding a disk in $L$ without changing the intersection number with $R_1$, thereby obtaining the desired sphere. Assume the conclusion holds for $r-1$. Without loss of generality, we may assume
		\begin{align*}
			|a_1| = min\lbrace |a_1|,|a_2|,\dots,|a_r|\rbrace.
		\end{align*}
		If $a_1 = 0$, we fill $\gamma_1$ with a disk in $L$, reducing the problem to the case of $r-1$ curves without changing the homology class of $U$. If $a_1\ne 0$, we may assume that $a_2$ satisfies $a_1a_2<0$. We select a segment of $\gamma_2$ corresponding to $-ka_1\beta$ (possible since $|a_2|\geq|a_1|$) and connect its endpoints by a curve $\Gamma$ that winds around $\gamma_1$ exactly once (as shown by the dotted line in Figure \ref{figure2}). It is clear that the region enclosed by $\gamma_1$, $-ka_1\beta$ and $\Gamma$ is a cylinder, which implies that $-\gamma_1$ and the composition of $-ka_1$ and $\Gamma$ are freely homotopic:
		\begin{align*}
			\lbrack -ka_1\beta\rbrack*\lbrack\Gamma\rbrack = \lbrack-\gamma_1\rbrack = \lbrack -ka_1\beta\rbrack \in\lbrack S^1, A_1\rbrack.
		\end{align*}
		Therefore, as elements of $\pi_1(A_1)$, $\lbrack -ka_1\beta\rbrack*\lbrack\Gamma\rbrack$ and $\lbrack -ka_1\beta\rbrack$ satisfy the conjugate relation:
		\begin{align}\label{405}
			\lbrack -ka_1\beta\rbrack*\lbrack\Gamma\rbrack = g^{-1}\lbrack -ka_1\beta\rbrack g \in \pi_1(A_1).
		\end{align}
		for some $g\in \pi_1(\mathcal{W})$. Note that $A_1 = (\mathcal{V}-D^3)\times S^1$, $\beta$ must be in the center of $\pi_1(A_1) = \pi_1(\mathcal{V}-D^3)\times\pi_1(S^1)$. It follows from (\ref{405}) that $\Gamma$ is null-homotopic.
		
		Therefore, we can fill $\Gamma$ by a pair of disks with opposite orientations, which decompose $U$ into a cylinder, and a punctured disk with $r-1$ boundary components. Note that such an operation does not change the homology class of $U$, and the conclusion follows from the inductive hypothesis. (Figure \ref{figure2} gives an example that $a_1=3$, $a_2=-4$. Each straight segment denotes $k\beta$, and each angular point is mapped to the base point $p$).
	\end{proof}

	\begin{lemma}\label{lem3-4}
		Let $h:T^2 \longrightarrow T^3 = S^1_1\times S^1_2\times S^1_3$ be a map, $\alpha_1,\alpha_2$ the generators of $H_1(T^2)$ and $\beta_1,\beta_2,\beta_3$ the generators of $H_1(T^3)$. If $h_*(\alpha_1) = k\beta_1$, then the intersection number of $h$ and $S^1_3$ is a multiple of $k$.
	\end{lemma}
	\begin{proof}
		We compute
		\begin{align*}
			&\lbrack h\rbrack\cdot\lbrack S^1_3\rbrack = D_{T^3}(\beta_3)(h_*\lbrack T^2\rbrack) = (\hat{\beta}_1\smallsmile \hat{\beta}_2)(h_*\lbrack T^2\rbrack) = (h^*\hat{\beta}_1\smallsmile h^*\hat{\beta}_2)(\lbrack T^2\rbrack)
		\end{align*}
		where $\hat{\alpha}_i,\hat{\beta}_i$ denote the dual basis of $\alpha_i$ and $\beta_i$, and $D_{T^3}$ denotes the Poincare Dual map of $T^3$. Since $h_*(\alpha_1) = k\beta_1$, we have $h^*(\hat{\beta}_1) = k\hat{\alpha}_1 + c_{12}\hat{\alpha}_2$ and $h^*(\hat{\beta}_2) = c_{22}\hat{\alpha}_2$. Hence
		\begin{align*}
			&\lbrack h\rbrack\cdot\lbrack S^1_3\rbrack
			= kc_{22}(\hat{\alpha}_1\smallsmile\hat{\alpha}_2)(\lbrack T^2\rbrack).
		\end{align*}
		This gives the desired conclusion.
	\end{proof}
	
	\begin{lemma}\label{lem3-5}
		Let $U$ be a region in $S^2$ and $f:U\longrightarrow A_1$ a map. If the boundary components of $U$ are all mapped to $k$ times the generator of $\pi_1(L) = \pi_1(S^2\times S^1) = \mathbb{Z}$, then the intersection number of $f_*([U,\partial U])$ and $R_1$ is a multiple of $k$.
	\end{lemma}
	\begin{proof}
		By Lemma \ref{lem3-4}, $f:(U,\partial U)\longrightarrow (A_1,\partial A_1)$ is homologous to the sum of some cylinders $C_1,C_2,\dots,C_l$ and at most one sphere $C_0$. For $i\ge 1$, the boundary of each $C_i$ represents the same element in $\pi_1(L) = \pi_1(S^2\times S^1) = \mathbb{Z}$ and can thereby be connected to each other by a small cylinder in $L$, which forms a torus $T_i$. Moreover, it is straightforward to see that the intersection number of $T_i$ and $R_1$ equals that of $C_i$ and $R_1$, and the meridian of $T_i$ is mapped to $ka_i$ times the generator of $\pi_1(L)$. 
		
		Based on $j$ we can define
		\begin{align*}
			J: (\mathcal{V}-D^3)\times S^1\longrightarrow \mathcal{V}\times S^1 \stackrel{j\times id}{\longrightarrow} \Sigma_0\times S^1.
		\end{align*}
		Then $R = j^{-1}(S^1_1)\times \{1\} = J^{-1}(S^1_1\times \{1\})$. By the Poincare Duality we have
		\begin{align*}
			\lbrack T_i\rbrack\cdot\lbrack R\rbrack = \lbrack T_i\rbrack\cdot\lbrack J^{-1}(S^1_1)\rbrack = J_*\lbrack T_i\rbrack\cdot\lbrack S^1_1\rbrack.
		\end{align*}
		On the other hand, $J_*\lbrack T_i\rbrack$ satisfies the condition of Lemma \ref{lem3-4} and yields a $k$-multiple in the intersection number. By the same reason we also have
		\begin{align*}
			\lbrack C_0\rbrack\cdot\lbrack R_1\rbrack = J_*\lbrack C_0\rbrack\cdot\lbrack S^1_1\rbrack
		\end{align*}
		where the right-hand side equals $0$ as $\pi_2(T^3)=0$. This completes the proof of the Lemma \ref{lem3-5}.
	\end{proof}
	
	\begin{proof}[\textbf{Proof of Lemma \ref{lem3-2}}]
		The conclusion follows from the paragraph before Lemma \ref{lem3-3}, and the result of Lemma \ref{lem3-5}.
	\end{proof}
	
	\begin{lemma}\label{lem3-6}
		Suppose the fiber of $\mathcal{W}$ is incompressible in $\pi_1(\mathcal{W})$. Let $\sigma$ be an immersed sphere in $\mathcal{W}$, then the intersection number of $\sigma$ and $R_1$ equals 0.
	\end{lemma}
	\begin{proof}
		The proof follows the same approach as in Lemma~\ref{lem3-2}. After decomposing the homology class into cylinders, we observe that the boundary of each cylinder must be contractible in $L$ due to the incompressibility condition. We can then complete each cylinder into a sphere in $A_1$ by attaching a pair of disks with opposite orientations in $L$. This operation preserves the intersection number with $R_1$. The conclusion follows by applying the same argument as in the proof of Lemma~\ref{lem3-5}, combined with the fact that $\pi_2(\Sigma_0\times S^1) = \pi_2(T^3) = 0$.
	\end{proof}
	
	With these preparation, we now carry out the proof of Theorem \ref{thm4}.
	\begin{proof}[\textbf{Proof of Theorem \ref{thm4}}]

		We first deal with the case that $X$ is a $K(\pi,1)$ 3-dimensional manifold. By Lemma \ref{lem3-2} and Lemma \ref{lem3-6}, combined with the fact that $T$ and $R_1$ has the intersection number 1, and the assumption that the fiber of $E$ is homotopically non-trivial, we conclude that the homology class $\lbrack T\rbrack$ is non-spherical. Since $T = R_2^+\cap \Sigma_0\times S^1$ by (\ref{411}), we conclude that $\mathcal{W}$ is \textit{SYS}. If the bundle $E$ in Theorem \ref{thm4} admits a PSC metric, then the width of $\mathcal{W}$ has a uniform upper bound. However, by (\ref{404}) the width of $\mathcal{V}$ can be arbitrarily large. It follows from our pullback construction that the width of $\mathcal{U}$ and $\mathcal{W}$ can also be arbitrarily large (compared with \eqref{33}\eqref{34}). This contradicts Theorem \ref{thm7}. 
		
		To see the conclusion holds true for manifolds with degree 1 maps to $E$, we observe that any band admitting a degree 1 map to an \textit{SYS} band is \textit{SYS} (see \cite{ref7}, Sec.5). For the rigidity part, we note that as a consequence of \cite{CG}, the universal covering of a Ricci flat but non-flat compact 4-manifold must also be compact. The conclusion follows as $E$ has a non-compact universal covering space.
		
		Now let $X$ be a Cartan-Hadamard manifold of dimension $n\le 6$. We take $\Tilde{X}$ to be the universal covering of $X$ and $\Tilde{B}$, $\Tilde{E}$ as above. By the same argument as in the proof of Theorem \ref{thm2}, there exists a torical band $\mathcal{V}\cong T^{n-1}\times \lbrack 0,1\rbrack$ in $\Tilde{X}$, which can be prescribed arbitrarily long. Construct $\mathcal{U}$ and $\mathcal{W}$ as in (\ref{402})(\ref{403}), we have
		\begin{align*}
			&\mathcal{U} = \mathcal{V}\#_{p_1}M\#_{p_2}M\#\dots\#_{p_k}M = \mathcal{V}\#N,\\
			&\mathcal{W} = P\#_{p_1\times S^1}Z\#_{p_2\times S^1}Z\#\dots\#_{p_k\times S^1}Z = P\#_{S^1}Y,
		\end{align*}
		where $P = T^{n-1}\times \lbrack 0,1 \rbrack\times S^1$. We write
		\begin{equation}
			\begin{split}
				P &= T^{n-1}\times\lbrack 0,1\rbrack\times S^1\\
				&= S^1_1\times S^1_2\times\dots\times S^1_{n-1}\times\lbrack 0,1\rbrack\times S^1_n.
			\end{split}
		\end{equation}
		
		Let
		\begin{equation}
			\begin{split}
				&Q_0 = S^1_1\times S^1_2\times\dots\times S^1_{n-1}\times\lbrace 0\rbrace\times S^1_n,\\
				&Q_i = S^1_1\times S^1_2\times\dots\times \hat{S}^1_i\times\dots\times S^1_{n-1}\times\lbrack 0,1\rbrack\times S^1_n, \\
				&i = 1,2,\dots,n-2,\\
				&R = S^1_1\times S^1_2\times\dots\times S^1_{n-2}\times\lbrack 0,1\rbrack,\\
				&T = S^1_{n-1}\times\lbrace 0\rbrace\times S^1_n.
			\end{split}
		\end{equation}
		Since $Q_i$ is of codimension 1 in $P$ and contains the $S^1_n$ factor, we can perturb it slightly to avoid $q_1\times S^1_n,q_2\times S^1_n,\dots,q_r\times S^1_n$. This gives a natural embedding of $Q_i,R,T$ into $\mathcal{W}$. For notational simplicity, the embedded images of $Q_i,R,T$ are still denoted by themselves. Clearly,
		\begin{align}\label{47}
			T = Q_0\cap Q_1\cap\dots\cap Q_{n-2}.
		\end{align}
		Since $Q_0$ and $Q_i\quad (i\ge 1)$ can be represented by homology classes in $H_n(\mathcal{W})$ and $H_n(\mathcal{W},\partial\mathcal{W})$, we can denote their Poincare Dual by $\beta \in H^1(\mathcal{W},\partial\mathcal{W})$ and $\alpha_i\in H^1(\mathcal{W})$. Therefore, it follows from (\ref{47}) that
		\begin{align}
			\lbrack T\rbrack = \lbrack\mathcal{W},\partial\mathcal{W}\rbrack\smallfrown\alpha_1\smallfrown\alpha_2\smallfrown\dots\smallfrown\alpha_{n-2}\smallfrown\beta\in H_2(\mathcal{W})
		\end{align}
		where $\lbrack T\rbrack$ is the homology class represented by $T$. 
		
		Following the same argument in the proof of Lemma \ref{lem3-2} and Lemma \ref{lem3-6}, we obtain that the intersection number of any 2-dimensional spherical class in $\mathcal{W}$ and $R$ is a multiple of $k$, when the fiber of $E$ is of order $k$ in the fundamental group, or equals $0$ when the fiber is incompressible. To be more precise, we may take the projection map $j: T^{n-1}\times\lbrack 0,1\rbrack\longrightarrow \Sigma_0 = T^{n-1}$, which serves as the counterpart of the map '$j$' in (\ref{j}) in the 3-dimensional case. Note that $\lbrack T\rbrack\cdot\lbrack R\rbrack = 1$. Combined with our assumption that $k>1$, we see that $\lbrack T\rbrack$ is non-spherical. Therefore, $\mathcal{W}$ is a \textit{SYS} band, and the conclusion follows.
		
		For the rigidity part, let $g_E$ be a Ricci flat metric on $E$. It follows from Lemma \ref{lem10} that $(\Tilde{E},g_{\tilde{E}})$ has the volume growth rate at least $n$. Denote the universal covering of $E$ by $\hat{E}$. By the Cheeger-Gromoll splitting theorem \cite{CG} we have $\hat{E} = W^s\times \mathbb{R}^{n+1-s}$, where $W$ is a closed simply connected Ricci flat manifold. It follows from Lemma \ref{lem8} that
		\begin{align*}
			n+1-s\ge r(\hat{E})\ge r(\Tilde{E})\ge n.
		\end{align*}
		Therefore, we see that $s=0$ and $\hat{E}$ is isometric to the Euclidean space. This completes the proof of Theorem \ref{thm4}.
	\end{proof}
	
	\begin{remark}
		Our method can also be used to derive some results on non-compact manifolds. Let $X$ be a $K(\pi,1)$ 3-manifold (not necessarily compact) with $\pi_1(X)\ne 0$. Following our strategy in the proof of Theorem \ref{thm4}, one is able to show that the $S^1$ bundle over $X\#_{p_1}M_1\#_{p_2}M_2\#\dots$ does not admit any uniformly PSC metric, where $p_i$ ($i = 1,2,\dots)$ form a discrete point set of $X$ and $M_i$ are closed manifolds.
	\end{remark}
	
	Employing an analogous approach, we can establish a PSC obstruction result for $S^1$ bundles over $1$-enlargeable manifolds.
	\begin{definition}
		Let $X$ be an $n$-dimensional closed manifold. We say $X$ is $1$-enlargeable, if there exists a metric $g$ on $X$ (or equivalently, for each metric $g$ on $X$), such that for each $\epsilon>0$, there exist a covering $(\tilde{X},\tilde{g})$ of $(X,g)$ and a degree-$1$ map $f:(\tilde{X},\tilde{g})\longrightarrow S^n(1)\subset\mathbb{R}^{n+1}$ which is constant outside a compact set, satisfying $\mathrm{Lip}(f)<\epsilon$.
	\end{definition}
	\begin{proposition}\label{thm3}
		Let $B$ be an $n$-dimensional 1-enlargeable manifold and $E$ be an $S^1$ bundle over $B$ ($n+1\le 7$). If the fiber of $E$ is homologically non-trivial, then $E$ carries no PSC metric.
	\end{proposition}
	
	In fact, we have the following more general result.
	\begin{proposition}\label{thm3'}
		Let $f: Y^{n+1}\longrightarrow X^n$ be a map, where $X^n$ is 1-enlargeable, and the homological pullback of a single point in $X$ under $f$ is nontrivial in $H_1(Y)$. Then $Y$ does not admit PSC metrics.
	\end{proposition}
	\begin{proof}
		Since $X$ is 1-enlargeable, we know that for any given $\epsilon>0$, there exist a covering $\Tilde{X}$ of $X$ and a degree 1 map $h:\Tilde{X}\longrightarrow S^n\subset \mathbb{R}^{n+1}$ such that $Lip(h)<\epsilon$. By pulling back a fixed torical band of width $\delta(n)$ in $S^n$ via $h$, we obtain a 1-overtorical band $\mathcal{U}$ in $\Tilde{X}$ of width $\frac{\delta(n)}{\epsilon}$. By considering the pullback diagram, we have
		\begin{align*}
			\begin{CD}
				Y^{n+1}     @>f>>  X\\
				@AAA        @AAA\\
				\Tilde{Y}     @>\Tilde{f}>> \Tilde{X}
			\end{CD}
		\end{align*}
		where $\Tilde{Y}$ is a covering space of $Y$ and $\Tilde{f}$ is a proper map. By perturbing $\Tilde{f}$ slightly, we may assume that it transversally intersects $\mathcal{U}$. Let $\mathcal{V} = \Tilde{f}^{-1}(\mathcal{U})$. For a generic point $p\in\mathcal{U}$, we denote the homological pullback of $p$ under $\Tilde{f}$ by $\Tilde{f}_!(p)$,$i.e.$, the homology class represented by $f^{-1}(p)$. Since $\mathcal{U}$ is 1-overtorical, by pulling back the fundamental class of $T^{n-1}\times \lbrack 0,1\rbrack$ we obtain
		\begin{align}
			\lbrack \mathcal{U},\partial\mathcal{U}\rbrack^* = \alpha_1\smallsmile\alpha_2\smallsmile\dots\smallsmile\alpha_{n-1}\smallsmile\beta
		\end{align}
		where $\alpha_i\in H^1(\mathcal{U})$ and $\beta\in H^1(\mathcal{U},\partial\mathcal{U})$. It follows that
		\begin{align*}
			&\Tilde{f}_!(p)\\ 
			= &D_{(\mathcal{V},\partial\mathcal{V})}\circ f^*\circ D_{(\mathcal{U},\partial\mathcal{U})}^{-1}(p)\\
			= &D_{(\mathcal{V},\partial\mathcal{V})}\circ f^*(\alpha_1\smallsmile\alpha_2\smallsmile\dots\smallsmile\alpha_{n-1}\smallsmile\beta)\\
			= &\lbrack\mathcal{V},\partial\mathcal{V}\rbrack\smallfrown f^*\alpha_1\smallfrown f^*\alpha_2\dots\smallfrown f^*\alpha_{n-1}\smallfrown f^*\beta.\\
			\ne &0
		\end{align*}
		Here $D_{(\mathcal{V},\partial\mathcal{V})}$ denotes the Poincare Dual map for $(\mathcal{V},\partial\mathcal{V})$ and $\lbrack \mathcal{U},\partial\mathcal{U}\rbrack^*$ represents the generator of top dimensional relative cohomology class. From (\ref{401}) we see that $\mathcal{V}$ is \textit{SYS}.
		
		On the other hand,
		\begin{align*}
			width(\mathcal{V})\ge \frac{1}{Lip\Tilde{f}}width(\mathcal{U}) = \frac{1}{Lipf}width(\mathcal{U})\ge \frac{1}{Lipf}\frac{\delta(n)}{\epsilon}.
		\end{align*}
		In conjunction with Theorem \ref{thm7}, we conclude that $Y$ admits no PSC metric.
	\end{proof}
	
	\begin{remark}
		In \cite{ref7}, Gromov defined the class of \textit{SYS}-enlargeable manifolds which includes all manifolds of infinite \textit{SYS} width. In this subsection, we actually produce new classes of \textit{SYS}-enlargeable manifolds.
	\end{remark}
	
	We conclude this subsection with a proof of Corollary \ref{cor2}.
	\begin{proof}[\textbf{Proof of Corollary \ref{cor2}}]
		If the Hurewicz map on $\pi_2(B)$ vanishes, by Lemma \ref{A1}, the boundary map $\partial$ in the exact sequence
		\begin{align*}
			\pi_2(B)\stackrel{\partial}{\longrightarrow}\pi_1(S^1)\stackrel{i}{\longrightarrow}\pi_1(E)\longrightarrow\pi_1(B)
		\end{align*}
		vanishes, so the fiber must be incompressible in $E$. Hence, the conclusion follows directly from Theorem \ref{thm3.5}. 
		
		If $B$ is of dimension 3 and contains no $S^2\times S^1$ factor, then each prime factor of $B$ has trivial $\pi_2$. Combined with Lemma \ref{lem11} one concludes the restriction of $E$ on each prime factor has incompressible fibers, which implies that $E$ has incompressible fibers. The conclusion then follows from Theorem \ref{thm4}.
	\end{proof}
	
	\subsection{Examples}
	
	To give a better interpretation of how the topological non-trivial condition of the $S^1$ fiber affects the geometry of the bundle, we would also like to include some examples here.
	
	We begin with the simplest case, $T^2$. Even in this setting, there exists an $S^1$-bundle over $T^2$ whose fiber is zero-homologous. It suffices to take those with non-divisible Euler class. In comparison, the fiber is always \textit{incompressible} because $\pi_2(T^2) = 0$.
	
	A similar example is $T^n \# T^n$, which has the twisted $S^1$ stability property by Theorem \ref{cor2}. Additionally, one can observe that the Schoen-Yau descent argument can hardly be applied in these cases, as it requires a homologically free fiber---a condition satisfied only when the bundle is trivial, as an immediate consequence of Lemma \ref{A4}.
	
	There also exist examples with compressible but homologically nontrivial fibers. Consider the $S^1$ bundle over $T^4\#\mathbb{C}\mathbb{P}^2$, whose restriction on $T^4$ is trivial, while behaving like $S^1\longrightarrow S^5\slash\mathbb{Z}_p\longrightarrow \mathbb{C}\mathbb{P}^2$ on $\mathbb{C}\mathbb{P}^2$. Clearly, it admits no PSC metric by Theorem \ref{thm3}. Moreover, if the restriction of the bundle on $T^4$ has non-divisible Euler class, then the fiber is forced to be zero-homologous, but still homotopically non-trivial provided $p>1$. In this case, Theorem \ref{thm4} is the only result to be applied.
	
	We conclude this section by providing an example to show that even the homologically nontrivial condition for $S^1$ fibers is not always a sufficient condition to rule out PSC metrics on fiber bundles. This phenomenon can happen only when the dimension of the total space is greater than or equal to $4$, exhibiting a tremendous difference compared with the dimension $3$ case, as demonstrated by Theorem \ref{thm2}.
	\begin{example}\label{eg: fail in n ge 4}
		
		$\quad$
		
		\begin{itemize}
			\item Let $\bar{E}$ be an $S^1$ bundle over the $K_3$ surface, such that the divisibility of the Euler class of $\bar{E}$ is $3$. Due to Lemma \ref{A1}, $\bar{E}$ is a spin 5-manifold with $\pi_1(\bar{E}) = \mathbb{Z}_3$. By \cite[Theorem 2.12]{ref17}, $\bar{E}$ carries a PSC metric.
			\item For $n\ge 4$, let $B_n = K_3\times T^{n-4}$ and $\bar{E}_n = \bar{E}\times T^{n-4}$. It follows that $\bar{E}_n$ is an $S^1$ bundle over $B_n$ with a homotopically nontrivial fiber. Since $B_n$ is $\hat{A}$-enlargeable in the sense of \cite{ref6}, it carries no PSC metric. This shows that the twisted $S^1$ stability property fails for $B_n$.
		\end{itemize}
	\end{example}
	\section{Codim 2 PSC obstruction in twisted settings}
	As an application of our result, we give  a partially affirmative answer to Gromov's quadratic decay question on non-trivial bundles in \cite{ref7}. 
	
	\begin{proof}[\textbf{Proof of Theorem \ref{thm9}}]
		Let $V$ be a small tubular neighborhood of the zero section of $X$, then $\partial V$ is an $S^1$ bundle over $B$. By Definition \ref{defn: twisted S^1 property}, we know that any closed manifold admitting a degree 1 map to $\partial V$ cannot support a PSC metric. Since $X\backslash V$ is diffeomorphic to $\partial V\times\lbrack 0,+\infty)$, we obtain the natural projection map
		\begin{align}\label{2}
			p: X\backslash V\longrightarrow \partial V.
		\end{align}

		Let $f:Y\longrightarrow X$ be a degree 1 map. Perturb $f$ such that it intersects transversally with $\partial V$. Let $V'=f^{-1}(V)$, then $S=\partial V'=f^{-1}(V)$ is an embedded submanifold of $Y$. Define
		\begin{align*}
			U_R = \lbrace x\in Y\backslash V', d_S(x)\le R\rbrace
		\end{align*}
		where $d_S$ denotes a smooth approximation function of the distance function to $S$. Let $S_R=\partial U_R\backslash S$. For any hypersurface $\Sigma$ in $U_R$ separating $S$ and $S_R$, let $W$ be the region bounded by $S$ and $\Sigma$. By (\ref{2}) we have
		\begin{align}\label{3}
			deg(p\circ f|_{\Sigma}) = deg(p\circ f|_S) = degf = 1.
		\end{align}
		Since $\partial V$ has the dominated twisted $S^1$ stability property, we know that $\Sigma$ admits no PSC metric.

		By a standard $\mu$-bubble argument (see \cite{ref23}), the width estimate holds for the band $U_R$. Specifically, if $U_R$ has scalar curvature bounded below by $\sigma$, then its width cannot exceed $\frac{2\pi}{\sqrt{\sigma}}$. 
		
		Since the width of $U_R$ is at least $R$, we obtain the curvature bound:
		\begin{align}\label{eq:curvature_bound}
			\inf_{x\in U_R} \mathrm{Sc}(x) \leq \frac{4\pi^2}{R^2}.
		\end{align}
		
		Letting $R_0$ denote the diameter of $V'$, we immediately obtain inequality \eqref{1} from \eqref{eq:curvature_bound}.
	\end{proof}
	
	Note that all the manifolds in Corollary \ref{cor2} have the \textit{dominated twisted $S^1$ stability property}, and hence Theorem \ref{thm9} applies. It follows from Theorem \ref{thm4} and a same argument as above that for any 3-manifold which does not admit PSC metric, the quadratic decay inequality in Theorem \ref{thm9} also holds under the homotopically nontrivial condition of the fiber.
	
	With the benefit of PSC obstruction on non-trivial bundles, we can also prove the following result:
	
	\begin{proposition}\label{thm10}
		Let $X^n$ ($n\le 7$) be a non-compact manifold with $B^{n-2}$ as a codimension 2 deformation retract, where $B$ has the \textit{dominated twisted $S^1$ stability property}. Then there exists a constant $R_0$, such that
		\begin{align*}
			\inf_{x\in B(R)}Sc(x)\le \frac{4\pi^2}{(R-R_0)^2}.
		\end{align*}
	\end{proposition}
	\begin{proof}
		Let $F:X\longrightarrow B$ be the retraction map, and let $V$ be the tubular neighborhood of $B$, then $\partial V$ is diffeomorphic to an $S^1$ bundle over $B$. By Lemma \ref{A5}, $F|_{X\backslash V}$ can be lifted to $\Tilde{F}:X\backslash V\longrightarrow \partial V$, with $\Tilde{F}|_{\partial V}$ homotopic to the identity. This implies that any hypersurface separating $\partial V$ from the infinity admits a degree 1 map to $\partial V$, and the conclusion follows from an argument analogous to that in the proof of Theorem \ref{thm9}.
	\end{proof}
	
	By the same argument, we obtain the proof of Corollary \ref{cor3}.
	
	$\quad$

	\textbf{Acknowledgements} This work is supported by National Key R\&D Program of China 2020YFA0712800. I am deeply indebted to my advisor Prof. Yuguang Shi for suggesting this problem, constant support and many inspiring conversations. I am grateful to Prof. Yi Liu and Prof. Yi Xie for their enlightening conversations on the geometrization conjecture and gauge theory. I wish to thank Dr. Jintian Zhu for many helpful discussions. I am also grateful to anonymous referees for their careful reading and suggestions which significantly improve the exposition of this paper.

	\appendix
	\section{Topological and algebraic preliminary}
	In this appendix, we prove several topological and algebraic facts used in the paper, including useful properties of $S^1$ bundles that are frequently applied in Section 4.
	
	\begin{lemma}\label{A1}
		Let $\pi:E\longrightarrow B$ be an $S^1$ bundle with the Euler class $e$. Then in the long exact sequence
		\begin{align*}
			\pi_2(B)\stackrel{\partial}{\longrightarrow}\pi_1(S^1)\stackrel{i_*}{\longrightarrow}\pi_1(E)\stackrel{\pi_*}{\longrightarrow}\pi_1(B)
		\end{align*}
		we have
		\begin{align*}
			\partial:\sigma\longmapsto e(\mathrm{hur}(\sigma))
		\end{align*}
		for $\sigma\in \pi_2(B)$.
	\end{lemma}
	\begin{proof}
		Pick a base point $b_0\in B$ and $x_0\in\pi^{-1}(b_0)$ . Let
		\begin{equation}
			\left\{
			\begin{aligned}
				&F:[0,1]\times [0,1] \longrightarrow B\\
				&F|_{\partial ([0,1]\times [0,1])} = b_0
			\end{aligned}
			\right.	
		\end{equation}
		be an element in $\pi_2(B,b_0)$, then $F$ can be lifted to a map $\Tilde{F}:[0,1]\times [0,1]\longrightarrow E$ satisfying
		\begin{equation}
			\begin{split}
				\Tilde{F}(J) &= x_0,\\
				\Tilde{F}(K)&\subset \pi^{-1}(b_0)
			\end{split}
		\end{equation}
		where $J = \partial([0,1]\times [0,1])\backslash\lbrack 0,1\rbrack\times\lbrace 1\rbrace$, $K=\lbrack 0,1\rbrack\times\lbrace 1\rbrace$. It follows that $\Tilde{F}|_K: K\longrightarrow \pi^{-1}(b_0)=S^1$ represents an element in $\pi_1(S^1)$ corresponding to the image of $F$ under $\partial$ (see \cite[Theorem 4.41]{ref11}). 
		
		To compute this element, we first view $F$ as a map to $\bar{E}$, the $\mathbb{R}^2$ vector bundle that induces $E$. We regard $B$ as the zero section of $\bar{E}$. Since $\bar{E}$ and $E$ are homotopic equivalent, we can lift $F$ to a map $\bar{F}:[0,1]\times[0,1]\longrightarrow\bar{E}$ such that $\bar{F}([0,1]\times [0,1]) = x_0$.
		
		By perturbing $\bar{F}$ while keeping it stationary on $\partial([0,1]\times [0,1])$, we may assume that it intersects transversally with the zero-section $B$ at points $p_1,p_2,\dots,p_k$. We observe that $p_i$ can be chosen sufficiently close to $K$. In fact, for a point $p_i$, we can join it to a small neighborhood of $K$ by a curve $c_i$, such that the only intersection of $c_i$ and $B$ is $p_i$. It's not hard to imagine that $p_i$ can be strapped to a neighborhood of $K$ along the curve $c_i$. Next, we remove the small disk around $p_i$ and construct a lifting $\tilde{F}$ of $F$ to $\bar{E}$ which does not touch the zero section. This gives a desired lifting as in (A.2). 
		
		Since $F$ intersects the zero section of $E$ transversally, we see that each point in the intersection contributes exactly once in the process of winding the $S^1$ fiber over $b_0$. Therefore, we conclude that the element in $\pi_1(S^1)$ is given by the oriented intersection number of $F$ and the zero section. The relation with the Euler class is a consequence of the Thom isomorphism, which completes the proof.
		
	\end{proof}
	
	\begin{lemma}\label{A4}
		Let $B^n$ be a closed oriented manifold with $H^2(B,\mathbb{Z})$ free, and let $\pi:E\longrightarrow B$ be an oriented $S^1$ bundle over $B$ with the Euler class $e$. Let $\tau$ be the homology class represented by the fiber in $H_1(E)$, then\\
		(1) $\tau=0$ if and only if $e$ is non-divisible;\\
		(2) $\tau$ is a $k$-torsion element if and only if $e=k\alpha$, where $\alpha$ is non-divisible;\\
		(3) $\tau$ is a free element if and only if $e=0$.
	\end{lemma}
	\begin{proof}
		
		Let $p$ be a point in $B$. It suffices to calculate the homology pullback $\pi_!(p)$ of $p$ under $\pi$, which is exactly the homology class represented by the fiber. We have
		\begin{align*}
			\pi_!(p) = &D_E\circ \pi^*\circ D_B^{-1}(p)\\
			= &D_E\circ \pi^*(\lbrack B\rbrack^*)\\
		\end{align*}
		where $D_B$ denotes the Poincare dual map of $B$ and $\lbrack B\rbrack^*$ denotes the generator of $H^n(B)$. We apply the Gysin sequence as follows:
		\begin{align}\label{406}
			H^{n-2}(B)\stackrel{\smallsmile e}{\longrightarrow}H^n(B)\stackrel{\pi^*}{\longrightarrow}H^n(E).
		\end{align}
		Since $D_E$ is an isomorphism, it suffices to study the element $\pi^*(\lbrack B\rbrack^*)\in H^n(E)$. If $e=0$, then $\pi^*$ in (\ref{406}) is injective, and hence $\pi_!(p)$ is free. If $e\ne 0$, we assume $e = k\alpha$ with $\alpha$ a non-divisible element in $H^2(B)$. By Corollary 3.39 in \cite{ref11}, there exists $\beta\in H^{n-2}(B)$, such that $\alpha\smallsmile\beta = \lbrack B\rbrack^*$. This shows that the image of the first map in (\ref{406}) equals $k\mathbb{Z}\subset\mathbb{Z}\cong H^n(B)$, and the conclusion follows.
	\end{proof}
	
	\begin{lemma}\label{A5}
		Let $X^n$ be a non-compact manifold with $B^{n-2}$ as a codimension 2 deformation retract, with the retraction map denoted by $F:X\longrightarrow B$. Let $V$ be the tubular neighborhood of $B$ and let $Z=\partial V$ be an $S^1$ bundle over $B$. Let $\pi:Z\longrightarrow B$ be the bundle projection map. Then $F|_{X\backslash V}$ can be lifted to a map $\Tilde{F}:X\backslash V\longrightarrow Z$, such that the restriction of $\Tilde{F}$ on $Z$ is homotopic to the identity.
	\end{lemma}
	\begin{proof}
		Let $f=F|_Z$. We first lift $f$ to a map $\Tilde{f}:Z\longrightarrow Z$, such that $f=\pi\circ\Tilde{f}$ and $\Tilde{f}$ is homotopic to the identity. Fix a Riemannian metric on $X$ such that $V=\lbrace x\in X| d(x,B)\le\delta\rbrace$ for some small number $\delta$. Define the homotopy $h_t:Z\longrightarrow V\hookrightarrow X$ by
		\begin{align*}
			(p,x)\mapsto (p,\delta tx).
		\end{align*}
		Here $p=\pi(x)$, and $\delta tx$ means the scaling by the coefficient $\delta t$ in the disk fiber over $p$ in the bundle $V$. Then we have
		\begin{align*}
			&h_0(x) = \pi(x),\\
			&h_1(x) = x.
		\end{align*}
		Therefore, $H_t = F\circ h_t$ satisfies
		\begin{align*}
			&H_0(x) = \pi(x),\\
			&H_1(x) = f(x).
		\end{align*}
		Since $\pi:Z\longrightarrow B$ has the homotopic lifting property, and $H_0 = \pi$ can be lifted to $id:Z\longrightarrow Z$, we obtain a homotopy $\Tilde{H}_t:Z\longrightarrow Z$ that satisfies
		\begin{align*}
			&\Tilde{H}_0 = id,\\
			&H_t = \pi\circ\Tilde{H}_t.
		\end{align*}
		Let $\Tilde{f} = \Tilde{H}_1$, then $f=H_1=\pi\circ\Tilde{H}_1=\pi\circ\Tilde{f}$. This shows that $\tilde{f}$ has the desired property.
		
		We next lift $F$ to a map $\Tilde{F}$, such that $\Tilde{F}|_Z = \Tilde{f}$. Since $S^1$ is a $K(\mathbb{Z},1)$ space, the only obstruction of this lifting lies in $H^2(X\backslash V, Z, \mathbb{Z})=H^2(X,V,\mathbb{Z}) = 0$. This shows $F$ has the desired lifting, thereby completing the proof of Lemma \ref{A5}.
	\end{proof}
	
	\begin{lemma}\label{A3}
		Let $H,K$ be subgroups of $G$. If $K$ has finite index in $G$, then $H\cap K$ has finite index in $H$.
	\end{lemma}
	\begin{proof}
		We write
		\begin{align*}
			A = \lbrace gK|gK\cap H\ne\emptyset\rbrace.
		\end{align*}
		For any $g_1,g_2\in gK\cap H$, we have $g_1K=g_2K$. It follows that $g_2^{-1}g_1\in K$, $g_2^{-1}g_1\in H\cap K$, and $g_1(H\cap K)=g_2(H\cap K)$. This yields a well-defined map:
		\begin{align*}
			F\colon(G\slash K)\cap A &\longrightarrow H\slash H\cap K,\\
			gK&\mapsto g_0(K\cap H),
		\end{align*}
		where $g_0$ is an element in $gK\cap H$. The map $F$ is evidently surjective. Combined with $A\subset G\slash K$ being finite, we see that $H\slash H\cap K$ is finite.
	\end{proof}
	
	\section{Volume growth estimates on fiber bundles}	
	
	In this appendix, we investigate the volume growth of the non-compact manifold $\Tilde{E}$, the covering space of $E$ in Theorem \ref{thm2}, induced by the universal covering $\Tilde{X}\longrightarrow X$, with $X$ a Cartan-Hadamard manifold. The result is used in the rigidity part of the proof of Theorem \ref{thm2}.

	\begin{lemma}\label{lem10}
		Let $(B^m,g_B)$ be a closed Riemannian manifold with $B^m = X^m\#M^m$, where $X$ is a Cartan-Hadamard manifold. Let $(E,g_E)$ be an $F$ bundle over $B$, with $F$ a compact manifold.  Let $\tilde{B}$ and $\tilde{E}$ be constructed as in \eqref{31} and \eqref{33}, then
		\begin{align*}
			r(\tilde{E},g_{\tilde{E}})\ge m,
		\end{align*}
		where $r$ is the volume growth rate function defined in Definition \ref{def4}.
	\end{lemma}
	\begin{proof}
		By the construction of $E$, we see that $G = \pi_1(X)$ acts properly and co-compactly on $\tilde{E}$. Let $S$ be a generator set of $G$ and $d_S$ be the corresponding word metric. By the Svarc-Milnor Lemma, we see that $(\tilde{E},d_{\tilde{E}})$ is quasi-isometric to $(G,d_S)$. 
		
		Let $\phi:(G,d_S)\longrightarrow (\tilde{E},d_{\tilde{E}}) $ and $\phi':(\tilde{E},d_{\tilde{E}})\longrightarrow (G,d_S)$ be the quasi-isometries, then we have constants $C,D>0$, such that
		\begin{align*}
			&C^{-1}d_S(x,y)-D<d_{\tilde{E}}(\phi(x),\phi(y))<Cd_S(x,y)+D,\\
			&C^{-1}d_{\tilde{E}}(z,w)-D<d_{S}(\phi'(z),\phi'(w))<Cd_{\tilde{E}}(z,w)+D
		\end{align*}
		hold for all $x,y\in G$ and $z,w\in \tilde{E}$. Consequently, for $R$ sufficiently large, it holds that
		\begin{align*}
			\phi(B(e,(2C)^{-1}R,G))\subset B(p,R,\tilde{E})\subset \phi(B(e,2CR,G)),
		\end{align*}
		where $e\in G$ is the unit and $p\in\tilde{E}$ represents a base point. Therefore, $f_1(R) = \mathrm{Vol}(B(p,R,\tilde{E},g_{\tilde{E}})$ is equivalent to the growth function $f_2(R)$ of $G$.
		
		Let $g_{CH}$ be a metric on $X$ with nonpositive sectional curvature and let $\tilde{g}_{CH}$ be the induced metric on $\tilde{X}$. With the same argument as above, we see that $f_3(R) = \mathrm{Vol}(B(q,R,\tilde{X},\tilde{d}_{CH})$ is also equivalent to the growth function $f_2(R)$ of $G$. The proof is completed by the volume comparison theorem.
	\end{proof}

\end{document}